\documentclass{birkjour}

\usepackage[toc,page]{appendix}
  \newtheorem{thm}{Theorem}[section]

 \theoremstyle{definition}
 \newtheorem{defn}[thm]{Definition}
 \theoremstyle{remark}
 
 \newtheorem*{ex}{Example}
 \numberwithin{equation}{section}

 \newtheorem{axiom}[thm]{Axiom}
\newtheorem{problem}[thm]{Problem}

\newcommand{\N}{\mathbb{N}}
\newcommand{\Z}{\mathbb{Z}}

\newcommand{\R}{\mathbb{R}}
\newcommand{\C}{\mathbb{C}}

\newcommand{\T}{\mathbb{T}}

\newcommand{\F}{\mathbb{F}}
\newcommand{\G}{\mathbb{G}}

\newcommand{\BS}{\mathbb{S}}
\newcommand{\BP}{\mathbb{P}}
\newcommand{\BH}{\mathbb{H}}

\begin{document}

\title{Examples of Morphological Calculus}
\author{Frank Sommen}
\address{Clifford Research Group,\\
	Department of Mathematical Analysis,\\
	Ghent University,\\
	Galglaan 2, B-9000 Gent, Belgium}
\curraddr{}
\email{fs@cage.ugent.be}



\date{}



\begin{abstract}
In this paper we present an introduction to morphological calculus in which geometrical objects play the rule of generalised natural numbers.
\end{abstract}

\maketitle
\tableofcontents

\section{Introduction}

Morphological calculus in an extension of the calculus of natural numbers $1,2,3$ etc whereby all sorts of geometrical objects are seen as generalised natural numbers. To make a list, we have
\begin{itemize}
\item the natural numbers $1,2,3,\ldots$
\item the real line $\R$
\item the set of natural numbers $\N$
\item carthesian spaces $\R^2,\R^3, \ldots$
\item projective spaces $\R \BP_n,\C \BP_n, \ldots$
\item Spheres $S^{n-1},\C S^{n-1}$
\item Groups like $SO(n),U(n), GL(n,\R),\ldots$
\item Gra{\ss}mann manifold  $G(m,k,\R)$
\end{itemize}
and other groups and homogeneous spaces. In fact any kind of geometrical objects can be added to the list.\\
The rules for morphological calculus extend the rules for calculating with natural numbers. We have
\begin{enumerate}
\item The addition $t_1+t_2+\cdots+t_k$\\
The terms $t_1,\ldots,t_k$ are supposed to represent morphological objects and the addition represents any object that can be formed by making a disjoint union of the objects $t_1,\ldots,t_k$ and glueing them together when possible. This glueing process is itself not part of the calculus so there in no unique way to do it and one also needn't do it; one can simply put the objects $t_1,\ldots,t_k$ in a list, as the language of calculus suggests.\\
For example $1+2+3+4$ can be visualised as a triangle of $10$ points: 1+(1+1)+(1+1+1)+(1+1+1+1).\\
The terms $t_1,\ldots,t_k$ in an addition may simply be names for morphological objects but, they also could be expressions between brackets like in:
$$
5+(3+1)+2+(1+2+7).
$$
The material between brackets is interpreted as a single morphological object.
\item The subtraction $t_1-t_2$\\
This means that the object $t_2$ is deleted from the object $t_1$. For example $3-2$ means to delete $2$ points from a set of $3$ points or $\R-1$ means to delete a point from a line. The subtraction  represents a problem: we have to look for an object such that $t_1-t_2=c$ or such that $t_1$ can be written as $c+t_2$. There may not be a morphologically acceptable solution for this. For example
$$
0=1-1
$$
means to create a point $1$ and then to wipe if off.\\
Also negative numbers like $-1,-2,\ldots$ are no objects of morphological calculus although they may be meaningful as actions: $-1=$ to delete one point, $-2=$ to delete two points, etc.\\
The subtraction is presented as a binary operation $t_1-t_2$ here, but of course one may also consider extended expression like $7-3-2+4-1$, as long as things add up to a morphological object.\\
\item The multiplication $v\cdot w$\\
For the natural numbers, the multiplication is a notation for repeated addition, so for example
\begin{itemize}
\item $1\cdot a=a$
\item $2\cdot a=a+a$
\item $3\cdot a=a+a+a$
\end{itemize}
etc. In other words, the meaning of multiplication is in fact determined by the rule of distributivity
$$
(t_1+t_2+\cdots+t_k)\cdot w=t_1\cdot w+t_2\cdot w+\cdots+t_k\cdot w.
$$
In morphological calculus, the product $v\cdot w$ means that every point of the object $v$ is replaced by a copy of $w$ and then all those copies of $w$ are possibly  glued together in some way that is not specified by the language of calculus.\\
Typical examples are: the carthesian product $v\times w$, a fibre bundle $E=M\cdot F$ with base space $M$ and fibre $F.$\\
One can also consider long multiplication like $v_1\cdot v_2\cdots v_k$ that may correspond to iterated fibre bundles. note that the fibre bundle interpretation is only an option; it isn't a must and it will not always be available.
\item The division $v/w$\\
Like the subtraction, also the division is seen as a problem: to find a morphological object $c$ for which $v=c\cdot w$. Any good solution to this will be denoted as $v/w$ and again there may not always be a solution. For example rational numbers like $1/2, 1/3, 4/7$ etc, are no morphological objects even though one may write $7/3=6/3+1/3=2+1/3.$
\end{enumerate}
Hence the language of morphological calculus is similar to that of natural numbers. There are however some aspects of language of calculus that cause dilemmas and also need more explanation
\begin{enumerate}
\item Names, definitions, substitutions.\\
Every morphological object has a name attached to it. For example $1,2,3,\ldots$ the natural members, $\R$ the real line and so on. Then every name is given a definition or several definitions of the form
$$
\mathrm{Name\;=\;Expression}
$$
the first main examples being the definitions of the natural numbers
$$
2=1+1,\; 3=1+1+1,\;4=1+1+1+1\;
$$
and so on.\\
Such definitions may come from geometry, but they are algebraic expression of some geometrical decomposition of an object, i.e., geometrical knowledge can only enter the calculus via algebraic relations of the form $\mathrm{Name\;=\;Expression}$.\\
When a name $N$ appears somewhere in an expression $E$, i.e., $E=E(N)$ and when one has a definition $N=Expr.$; then one may perform the substitution $E(N)=E((\mathrm{Expr.}))$, i.e., replacing the name $N$ by the expression $(Expr.)$ between brackets. Later on one may investigate how and when brackets may be removed. We do not use brackets in a redundant manner like e.g. $(7)$ is not used, $(\mathrm{Name})$ is not used $((\mathrm{Expr.}))$ is not used, $\mathrm{Name}=(\mathrm{Expr.})$ is not used.\\

Example: (The Fibonacci trees)\\
These morphological structures are defined by
$$
f_1,\;f_2=1,\;f_n=f_{n-1}+f_{n-2}
$$
leading to the solutions
\begin{gather*}
\begin{aligned}
&  f_2=1+1\\
&  f_3=(1+1)+1\\
&  f_4=((1+1)+1)+(1+1)\\
&  f_5=(((1+1)+1)+(1+1))+((1+1)+1),
\end{aligned}
\end{gather*}
so what appears here are not just the Fibonacci numbers $2,3,5,8,$ but the tree-like structures that give rise to these numbers if one removes the brackets. This tree-like structure is a typical example of a morphological object.

\item Commutativity, Associativity\\
In morphological calculus the addition $t_1+\cdots+t_k$ is in the first place a listing of objects; it is not commutative. Also within an addition one may consider expressions between brackets and since brackets refer to morphological objects one can't just ignore them; the addition is not just associative. On the other hand, for the natural numbers the addition also refers to the total quantity or sum. For example the total quantity of $5+(3+1)+2$ may be evaluated as:
\begin{gather*}
 \begin{aligned}
 5+(3+1)+2& =(1+1+1+1+1)+((1+1+1)+1)+(1+1)\\
&  =(1+1+1+1+1)+(1+1+1+1)+(1+1)\\
&  =1+1+1+1+1+1+1+1+1+1+1=11,
 \end{aligned}
\end{gather*}
so it requires substitutions $5\rightarrow(1+1+1+1+1)$ etc. and deleting the brackets. So the total quantity is evaluated within the language of calculus and not in some outside theory.  It corresponds to a morphological process in which the morphological structure is constantly changed to the extent that in the final evaluation of the quantity, the identity of the numbers $5,3,\ldots$ as well as their place in the context is lost. Commutativity, substitutions and putting and deleting brackets are guaranteed in so far that the total quantity is preserved, but they are also mutations. For more general morphological objects, such as the line $\R$ the notion of quantity is not defined and we will illustrate that, if it were defined it wouldn't correspond to the cardinality of a set.\\
Yet we calculate as if these objects would have a form of quantity and so, in particular, terms in an addition may be commuted, substituted and brackets may be put or deleted.\\
For the multiplication $v\cdot w$, commutativity $v\cdot w=w\cdot v$ is even less obvious especially if one thinks of a fibre bundle $E=M\cdot F$. But again these geometrical interpretations happen outside morphological calculus and the total quantity of $v\cdot w$ is the same as that of $w\cdot v$. Moreover, to be able to calculate one has to be able to commute factors in a product, even though this deforms the morphological structure. Also the law of distributivity
$$
(t_1+\cdots+t_k)\cdot w=t_1\cdot w+\cdots+t_k\cdot w
$$
is essential to give a meaning to the product while as the same time it is a deformation.\\
So, to conclude, the morphological universe consists of the totality of all meaningful algebraic expressions based on a set of names for morphological objects together with their definitions within calculus. The calculus rules, leading to the relations $A=B$ are the same as for the natural numbers and the relations $A=B$ are interpreted at the same time as morphological deformation and as preservation of quantity, whatever meaning this may have.
\end{enumerate}


\section{The Real Line}

The real line $\R$ is in mathematics defined as the set of all real numbers, represented as points on that line. It is hence an infinite point set and its cardinality $c$ is called the continuum; it is larger than the cardinality $\aleph_0$ of the natural numbers.\\
The real line decomposes as
$$
 \begin{aligned}
 & \R=\R_-\cup\{0\}\cup\R_+
 \end{aligned}
$$
with\newline
$\R_-=\;]-\infty,0[\;:\;$ the halfline of negative numbers.\\
\noindent $\R_+=\;]0,+\infty[\;:\;$ the halfline of positive numbers.\\
So $\R_+$ and $\R_-$ are open intervals that are closed off and glued together by the point $\{0\}$ to form the real line.\\
Morphologically we write this disjoint union as
$$
 \begin{aligned}
 & \R=\R_-+1+\R_+
 \end{aligned}
$$
whereby ``1'' represents the middle point $\{0\}$.\\
Both $\R_+$ and $\R_-$ are halflines having ``the same shape'', so we identify $\R_-=\R_+$, leading to the first definition
$$
 \begin{aligned}
 & \R=\R_++1+\R_+,
 \end{aligned}
$$
which, after commuting terms, leads to
$$
 \begin{aligned}
 & \R=2\R_++1.
 \end{aligned}
$$
Next one may argue that all open intervals $]a,b[$ have ``the same shape'', so they are all copies of $\R$ and, in particular, we may identify
$$
 \begin{aligned}
 & \R_+=\R,
 \end{aligned}
$$
leading to the relations
$$
 \begin{aligned}
 & \R=\R+1+\R=2\R+1.
 \end{aligned}
$$
This may be interpreted as the way to produce an open interval or curve $]a,c[$ by taking an open interval $]a,b[$, glue to it a point $\{b\}$ and then glue to the next open interval or curve $]b,c[$.\\
The question now is: what is the quantity of $\R$?\\
If it is the cardinality ``$c$'' then one should identify $\R+1$ with $\R$ but $\R+1$ would be a semi-interval like $]0,1]$, open from one side and closed from the other, which is not the same as $]0,1[$.\\
Next, the relation $\R=\R+1+\R$ indicates the fact that $\R$ contains at least one point and, by iteration
$$
\R=\R+1+\R=(\R+1+\R) + 1 + (\R+1+\R)=\cdots
$$
we obtain $3$ points, $7$ points, $15$ points etc., any finite number of points. So the morphological version of $\R$ seems to house infinity many points.\\
Now let us consider the relation
$$
\R=\R+1+\R=2\R+1
$$
as an equation. Then by subtracting $\R$ from both sides we get
$$
\R+1=0
$$
and by subtracting $1$ we get
$$
\R=0-1=-1,
$$
so that the total quantity of $\R$ should be $-1$.\\
This clearly conflicts with the idea of $\R$ being a set of points; the morphological line is hence not merely a set of points but rather a brand new object that doesn't quantify as a pointset. Of course one could argue that also
$$
\mathrm{infinity}=2\mathrm{infinity}+1,
$$
but infinity is a too trivial and vague number to work with for it absorbs everything.\\
There is an interesting interpretation for $\R=-1.$\\
Every manifold or surface of finite dimension may be represented by a cell complex, which we may represent by a polynomial
$$
a_o\R^{n}+a_1\R^{n-1}+\cdots+a_n,\;\; a_0\in\N, \;a_1,\ldots,a_n\in \N\cup\{0\}.
$$
By making the identification $\R=-1$ we obtain
$$
e\{M\}=a_0(-1)^n+a_1(-1)^{n-1}+\cdots+a_n,
$$
which is the Euler characteristic $e\{M\}$ of manifold $M$.\\
The Euler number $e\{M\}$ is a topological invariant and for a given manifold $M$ it is independent of the cell decomposition of that manifold. To see this, note that for any two cell decompositions of $M$ there exists a kind cell decomposition that refines both of them and so it suffices to consider the case $M=\R^n$. Moreover, every cell decomposition of $\R^n$ may be obtained from simple cell decompositions of the form $\R^j=2\R^j +\R^{j-1}$, which proves the invariance of $e\{M\}$ morphologically.

The fact that morphological calculus respects the Euler characteristic is like a corner stone (it is the final invariant that is preserved!). But as it is now, morphological calculus is reduced to the calculus of the integers $\Z$ and a point $1$ is identified with a closed interval $\R+2$ a plane $\R^2$ is identified with a point $1$.

Hence, the idea of a line as an infinite point set is completely lost and also the dimension of an object is not preserved. As a result we have the identification $\R+1=0$  between a semi-interval (or circle) $\R+1$ and the number zero and, in fact $\R=-1$ between a line $\R$ and the number $-1$, while numbers $0$ and $-1$ are no point sets and hence no objects.\\
To overcome this collapse of the notion of dimension we are going to introduce the following assumption.

\begin{axiom}
Morphological calculations are only granted if all the algebraic expressions and operations make sense in terms of geometrical objects.
\end{axiom}

Hence, in particular, number zero $0$ and negative numbers $-1,-2,$ etc. are hereby excluded or at least pushed to the background. Moreover, a relation like
$$
\R=\R+1+\R=2\R+1
$$
does not automatically allow one to solve it like an equation; one could also simply interpret it by stating that one is allowed to replace $\R$ by $2\R+1$ or vice-versa within calculations and nothing more. Hence, it does not automatically imply e.g. that $\R+1=0$ or even $\R-1\equiv2\R,$ although this last relation $\R-1=2\R$ makes morphological sense. This now leads to the following result.

\begin{thm}[Morphological Stability]
Under the assumption of the relation $\R=2\R+1,$ every cell complex $a_0\R^{n}+a_1\R^{n-1}+\cdots+a_0$ is equivalent to either
$$
a\R^n,\;\; a\in\N\;\;\; or \;\;\; \R^n+b\R^{n-1}, \;\; b\in \N,
$$
no further identifications being possible.
\end{thm}
\begin{proof}
The statement holds trivially for $n=0$. For $n=1$, $a>1$ and $b>0$ we clearly have
$$
a\R+b=(a-2)\R+2\R+(b-1)+1=(a-1)\R+b-1
$$
so we are reduced to either $b=0$ or $a=1.$\\
Next, assuming the property for $n-1$, $n>1$, we may reduce any cell complex to
$$
a\R^n+b\R^{n-1}+c\R^{n-2}.
$$
If $c=0$ we may reason as in the case $n=1$ to arrive at the final form $a\R^n$ or $\R^n+b\R^{n-1}.$ If $c>0$ we may write $a\R^n+b\R^{n-1}+c\R^{n-2}=(a-1)\R^n+\left(2\R^{n}+\R^{n-1}\right)+b\R^{n-1}+c\R^{n-2}=(a+1)\R^{n}+(b+1)\R^{n-1}+c\R^{n-2}$ and repeat this idea until $b>1.$ Then one may reduce using $2\R+1=\R:$ $
a\R^n+b\R^{n-1}+c\R^{n}=a\R^n+(b-1)\R^{n-1}+(c-1)\R^{n-2}$ and so on, until the final form is reached.
\end{proof}
Note that this theorem guarantees us that in the worst case, at least the
$$
\mathrm{dimension}=n
$$
as well as the
$$
\mathrm{Euler\;characteristic}=(-1)^na\;\; \mathrm{or}\;\; (-1)^n+(-1)^{n-1}b
$$
are being preserved during morphological calculations; it is a second approximation for any possible notion of morphological quantity (the first one being just the Euler characteristic).\\
But this calculus is still too poor and to be able to evaluate the quantity of $\R$ we have to ignore the distinction between a line $\R$  and a halfline $\R_+.$ This is again a dilemma, similar the once mentioned in introduction concerning commutativity and use of brackets. There are two options
\begin{enumerate}
\item The ``canonical'' option.\\
Hereby we assume as definition for $\R$ the relation
$$
\R=\R_++1+\R_+=2\R_++1
$$
and consider the identification $\R_+=\R$ as a form of decay. So the relation $\R=2\R+1$ is suspended in what we regard as ``the canonical style''.\\
This style of calculating is on the other hand flexible with respect to commutativity and the use of brackets. Its main purpose is (not exclusively): Morphological analysis: to analyse geometrical objects (surfaces, manifolds) by decomposing them into parts (or other ways) and to express this knowledge in calculus language in order to arrive at morphological definitions.\\
 \item The ``formal'' option.\\
 Hereby we consider morphological calculus as a formed language in which the order of terms in an addition and the use of brackets is not ignored. For the morphological line we have two definitions:
 $$
 \R=\R+1+\R\;\;\mathrm{or}\;\; \R=\R+(\R+1).
 $$
 Its main purpose is (not exclusively):\\
 Morphological synthesis: to construct a geometrical interpretation for an algebraic expression in morphological calculus.\\
 \end{enumerate}
 In this paper we mostly use the canonical style. Our main interest is to study manifolds and try to understand their morphological quantity, whatever that may mean. The formal style will be discussed briefly in last section.
\section{Carthesian Space, Spheres, Projective Spaces}
The carthesian plane is defined as the product $\R^2=\R\cdot \R,$ using the relation $\R=2\R_++1$ we thus arrive at
$$
\R^2=\left(2\R_++1\right)^2=4\R_+^2+4\R_++1,
$$
decomposing the plane into $4$ quadrants $\R_+^2,$ $4$ halfplanes $\R_+$ and one point $1$ (the origin).\\
Similarly the cartesian $n$-space is defined as product
$$
\R^n=\R\cdot \cdots\cdot R=\R\cdot\R^{n-1}
$$
and we have its decomposition into ``octants'':
$$
\R^n=\left(2\R_++1\right)^n=\sum_{j=0}^{n}\dbinom{n}{j}2^j\R_+^j.
$$
To define the sphere $S^{n-1}$ we make the following analysis: for a vector $\underline{x}\in\R^n$ with $\underline{x}\neq0$ we have the polar decomposition $\underline{x}=r\underline{\omega}$, $r=|\underline{x}|\in\R_+$, $\underline{\omega}=\displaystyle\frac{\underline{x}}{|\underline{x}|}\in S^{n-1}$.\\
In morphological calculus language we write
$$
\R^n-1=S^{n-1}\R_+,
$$
leading to the morphological definition of $S^{n-1}$:
$$
S^{n-1}=\frac{\R^n-1}{\R_+}.
$$
Now from $\R=2\R_++1$ we obtain that
$$
\R_+=\frac{\R-1}{2},
$$
a line without a point indeed gives $2$ halflines.\\
Hence we obtain a quantity formula for $S^{n-1}$:
$$
S^{n-1}=2\frac{\R^n-1}{\R-1}=2\R^{n-1}+2\R^{n-2}+\cdots+2\R+2.
$$
In particular, a circle is given by \newline
$
S^{1}=2\R+2\;:\:
$
two semi-circles and two points and a $2$-sphere is given by
$$
S^{2}=2\R^2+2\R+2=2\R^2+S^1,
$$
two hemi-spheres and a circle (equator). Also \newline
$S^2=S^1\R+2,$\newline
a ``cylinder $S^1\R$'' and two poles ``2''.\\
Using here $\R=2\R_++1$ we obtain
$$
S^{2}=2\left(2\R_++1\right)^2+2(2\R_++1)+2=8\R_+^2+12\R_++6,
$$
which may be interpreted as an octahedron whereby\newline $\R_+^2$ translates as a triangle,\newline $\R_+$ translates as a quarter circle or short interval.\\
Other regular polyhedra are harder to obtain, yet they are obtainable by transformations of the form  $\R=2\R+1,$ $\R^2=2\R^2+\R$, which as we know are questionable. In fact, every cell complex $a\R^2+b\R+c$ that corresponds to an embedded connected $2$-manifold in $\R^3$ has Euler characteristic $a-b+c=2(1-g)$, $g$ being the genus or number of holes, a number which characterises the manifold. Hence, for $2-$manifolds the relation $\R=2\R+1$ is not such a destructive deformation.\\
However, this also means that e.g. a dodecahedron will be identified with $12\R^2+30\R+20$ and hence a solid pentagon is identified with a square $\R^2$, an identification which is only topologically true.\\
For general spheres we have the recursion formula\newline
$S^{n-1}=2\R^{n-1}+S^{n-2}$: two hemi-spheres and an equator,\newline
as well as the ``polar coordinate'' formula\newline
$S^{n-1}=S^{n-2}\R+2$: a cylinder and $2$ poles.\\
These formula are special cases of the following general method for introducing polar coordinates on $S^{n-1}$.\\
Let $\underline{\omega}\in S^{n-1}$ and consider the decomposition $\R^n=\R^p\times \R^q$, $p+q=n$. Then we may write
$$
\underline{\omega}=\cos\theta\underline{\omega}_1+\sin\theta \underline{\omega}_2,\;\;\theta=\left[0,\frac{\pi}{2}\right], \underline{\omega}_1\in S^{p-1},\;\;\underline{\omega}_2\in S^{q-1}.
$$
There are three cases:
$$
\theta=0:\;\underline{\omega}=\underline{\omega}_1\in S^{p-1},\;\;\theta=\frac{\pi}{2}:\;\underline{\omega}=\underline{\omega}_2\in S^{q-1},
$$
$$
\theta\in\left]0,\frac{\pi}{2}\right[:\;\underline{\omega}=\cos\theta\underline{\omega}_1+\sin\theta\underline{\omega}_2\;\sim(\underline{\omega}_1,\underline{\omega}_2)\in S^{p-1}\times S^{q-1}.
$$
In morphological calculus this situation is expressed as follows:
$\R^n-1=\left(\R^p-1\right)\left(\R^q-1\right)+\left(\R^p-1\right)+\left(\R^q-1\right)$ or
$$
S^{n-1}\R^+=\left(S^{p-1}\R_+\right)\left(S^{q-1}\R_+\right)+\left(S^{p-1}\R_+\right)+\left(S^{q-1}\R_+\right)
$$
leading to the addition formula for spheres:
$$
S^{n-1}=S^{p-1}\cdot S^{q-1}\cdot \R_+ + S^{p-1} + S^{q-1}.
$$
Notice that also here $\R_+$ is interpreted as the quarter circle (small interval) $\theta\in\left]0,\frac{\pi}{2}\right[$, while the full line $\R$ would rather correspond to a semi-circle $\theta\in\left]0,\pi\right[$.\\
The addition formula also leads to:
\begin{gather*}
S^{n-1}=S^{p-1}\left(S^{q-1}\R_+ + 1\right)+S^{q-1}\\
=S^{p-1}\R^q+S^{q-1},
\end{gather*}
which generalises the recursion formula and the ``polar coordinate'' formula mentioned earlier.\\
Of particular interest is the odd-dimensional sphere $S^{2n-1}$ where we can take $p=q=n$.\\
This leads to the ``Hopf factorization formula''\newline
$$
S^{2n-1}=S^{n-1}\R^n+S^{n-1}\;\;\mathrm{or}\;\;S^{2n-1}=\left(\R^{n}+ 1\right)S^{n-1}.
$$
In particular  we have the Hopf fibrations
$$
S^3=S^2S^1,\;\;S^7=S^4S^3
$$
that are well known and follow from complex resp. quaternionic projective geometry. They can be seen as interpretations of the Hopf factorization
$$
S^3=\left(\R^2+1\right)S^1,\;\;S^7=\left(\R^4+1\right)S^3,
$$
whereby the spheres $S^2$ resp. $S^4$ are identified with $\left(\R^2+1\right)$ resp. $\left(\R^4+1\right)$. But of course the  Hopf fibrations are by no means proved or even implied by morphological calculus. \\
In general, the sphere $S^{n-1}$ can be mapped onto $\R^{n-1}$ by stereographic projection. Hereby one takes line from the southpole $\underline{w}=(0,\ldots,0,-1)$ to general point $\underline{w}$, denoted by $L(\underline{w})$ and the stereographic projection $\mathrm{st}(\underline{w})$ is the intersection of $L(\underline{w})$ with plane $x_n=1$ (the tangent plane to the north pole $(0,\ldots,0,+1)$).\\
This leads to the identification between $S^{n-1}$ and $\R^{n-1}\cup \{\infty\}$. In morphological calculus one might hence think of identification $S^{n-1}=\R^{n-1}+1$. But that would lead to the unwanted identification
$$
2\R^2+2\R+2=\R^2+1,
$$
that would again correspond to $\R=2\R+1$ via:
$$
\R^2+1=(2\R+1)\R+1=2\R^2+\R+1=2\R^2+(2\R+1)+1=2\R^2+2\R+2.
$$
In morphological calculus we introduce a kind of stereographic sphere or ``Poincar\'e sphere'' by $$\BS^n=\R^n+1,$$ leading to the recursion formula
$$
\BS^n=(2\R_+ + 1)\R^{n-1}+1=2\R^{n-1}\R_+ + \BS^{n-1}
$$
and leading to the total quantity (whatever that may mean)
$$
\BS^n=2\R^{n-1}\R_+ + 2\R^{n-2}\R_+ + \cdots +2\R_+ + 2.
$$
Notice hence that the identification $\R_+ =\R$ would lead to $\BS^n=S^n$ or $S^2=\R^2+1$ (a point and a square is a sphere). It is true that the only $2-$manifold interpretation for $\R^2+1$ is indeed a sphere. Also the Poincar\'e-polynomial of the sphere $S^n$ is given by $t^n+1$, which corresponds to $\R^n+1$.\\
Recall that the Poincar\'e-polynomial of a manifold $M$ is defined as $a_nt^n+\cdots+a_0$ with $a_j=\mathrm{dim}H_j,$ $H_j$ leading the $j-$th homology space of $M$.\\
It turns out that the Poincar\'e-polynomial often appears as the morphological quantity of an object, in particular for $\R^n$ itself and the sphere $\BS^n.$ But for the sphere $S^n$ we obtain a ``higher'' morphological quantity: $2\R^n+\cdots+2\R+2$ that does not correspond to the Poincar\'e-polynomial. The real projective space $\R\BP^n$ corresponds to the set of $1D$ subspaces of $\R^{n+1}$, also defined as the set of vectors $(x_1, \ldots,x_{n+1})\sim(\lambda x_1,\ldots,\lambda x_{n+1}),$ $\lambda\neq0.$\\
In mathematics we write it as the quotient structure
$$
\R\BP^n=\frac{\R^{n+1}\backslash \{0\}}{\R\backslash\{0\}}.
$$
This leads to the morphological definition
$$
\R\BP^n=\frac{\R^{n+1}-1}{\R-1}.
$$
and to the formula for the quantity of $\R\BP^n:$
$$
\R\BP^n=\R^n+\R^{n+1}+\cdots+\R+1.
$$
In this case the quantity polynomial corresponds to the Poincar\'e polynomial for $\R\BP^n:t^n+\cdots+t+1.$ It also leads to the recurson formula
$$
\R\BP^n=\R^n+\R\BP^{n-1}
$$
in which ``$\R^n$'' symbolizes the Affine subspace consisting of the points $(x_1,\ldots,x_n,1)$ while ``$\R\BP^{n-1}$'' stands for the plane at infinity: $x_{n+1}=0.$\\
Of course we also have that
$$
\R\BP^n=\frac{S^n}{S^0}=\frac{S^n}{2}
$$
whereby $S^0$ in the multiplicative group $S^0=\{-1,1\}$.\\
In particular the projective line is given by $$\R\BP^1=\R+1$$ symbolizing $\R\cup\{\infty\}$ and it also represents the Poincar\'e circle
$$
\BS^1=\R+1=S^1/2,
$$
$S^1=2\R+2$ being the standard circle.\\
The projective plane is given by
$$
\R\BP^n=\R^2+\R+1=\R^2+\R\BP^1=(\R+1)\R+1,
$$
whereby the object $(\R+1)\R$ in this context corresponds to a Moebius band.\\
Just seen by itself, $(\R+1)\R$ could correspond to several things, including any line bundle over the circle $\R+1$, i.e., either a cylinder or a Moebius band. In Geometry the Moebius band can be recognised by cutting it in half along the center circle; if it was a cylinder, then the cutted object would give $2$ cylinders and if it was a Moebius band then the cutted object would be a single cylinder. Now, this cutting procedure can be translated into morphological calculus as the subtraction
$$
(\R+1)\R-(\R+1)=(\R+1)(\R-1)=\R^2-1
$$
and $\R^2-1$ symbolizes a plane minus a point but also a single cylinder
$$
\R^2-1=S^1\R_+=(2\R+2)\R_+,
$$
here represented as a product of a circle $2\R+2$ (which has two glueing points and twice the length of the original circle) with a halfline $\R_+$ (stretching from the cutting point $\{0\}$ to the boundary $\{\infty\}$).\\
So this simple calculation symbolizes quite well the whole cutting experiment and it illustrates us the object $(\R+1)\R$ as being a Moebius band. In general, morphological objects are merely organised quantities that can have a number of meanings called morphological synthesis. This synthesis takes place outside the calculus but it can be guided by calculations that give the object an intrinsic meaning. In the Moebius experiment we also see that the circle $2\R+2$ and the Poincar\'e circle $\R+1$ clearly play different roles like also the line $\R$ and the halfline $\R_+$.\\
If we apply a similar experiment to the cylinders
$$
(2\R+2)\R-(2\R+2)=(2\R+2)(\R-1)=2S^1\R_+
$$
we obtain two cylinders. Of course one always calculates in a certain way and that may force a certain interpretation; the language of calculus can be used as an illustration but not as a real proof. In fact the language of calculus also has to remain flexible enough but this flexibility is at the cost of the stability of the morphological-synthesis. For example we have
$$
\R^2-1=S^1\R_+=(2\R+2)\R_+=2(\R+1)\R_+=2\BS^1\R_+
$$
showing that distributivity results in the cutting and reglueing of one cylinder $S^1\R_+$ into two cylinders $\BS^1\R_+$, half the size and with one single cutting edge $\R_+.$\\
For the general projective space we have a kind of ``Moebius factorization''
$$
\R\BP^n=\R\BP^{n-1}\cdot\R+1
$$
whereby the Moebius cutting experiment is represented  as
$$
\R\BP^{n-1}\cdot\R-\R\BP^{n-1}=\R\BP^{n-1}(\R-1)=\R^n-1=S^{n-1}\R_+,
$$
also a kind of cylinder.\\
We now turn to complex projective spaces.\\
The complex numbers $\C$ are morphologically given by $$\C=\R^2$$ and this is all. Anything concerning $\sqrt{-1}=i$ exists outside the calculus. We also have that
$$
\C^n=\left(\R^2\right)^n=\R^{2n}.
$$
Complex projective space $\C\BP^n$ is defined as the set of equivalence classes of relation
$$
(z_1,\ldots,z_{n+1})\in \C^{n+1}\setminus\{0\}\sim (\lambda z_1,\ldots,\lambda z_{n+1}), \;\; \lambda \in \C\setminus\{0\}
$$
i.e. the quotient structure
$$
\C\BP^n=\C^{n+1}\setminus\{0\}/\C\setminus\{0\}.
$$
Hence, in morphological calculus we have the definition
$$
\C\BP^n=\frac{\C^{n+1}-1}{\C-1}
$$
which immediately leads to the quantity
$$
\C\BP^n=\C^{n}+\C^{n+1}+\cdots+1=\R^{2n}+\R^{2n-2}+\cdots+1
$$
that also corresponds to Poicar\'e polynomial. The Euler number of $\C\BP^n$ equals $n$.\\
We also have that in real terms:
$$
\C\BP^{n-1}=\frac{\R^{2n}-1}{\R^2-1}=\frac{S^{2n-1}}{S^1},
$$
leading to the $\C\BP^n-$factorization of $S^{2n-1}$
$$
S^{2n-1}=\C\BP^{n-1}\cdot S^1,
$$
which is the fibration obtained from the group structure $(z_1,\ldots,z_{n})\in S^{2n-1}\rightarrow\left(e^{i\theta}z_1,\ldots,e^{i\theta}z_n\right)$.\\
In particular we have that
$$
\C\BP^1=\C+1=\R^2+1=\BS^2
$$
and the above fibration leads to the first Hopf-fibration
$$
S^3=\BS^2\cdot S^1.
$$
Like in the real case one has  the recursion formula
$$
\C\BP^n=\C^n+\C\BP^{n-1}
$$
and also the Moebius factorization
$$
\C\BP^n=\C\BP^{n-1}\cdot\C+1.
$$
Hereby the complex line bundle $\C\BP^{n-1}\cdot\C$ reduces for $n=2$ to
$$
\C\BP^2=\BS^2\cdot\C=(\C+1)\C
$$
and it is a non-trivial plane bundle over the $2-$sphere.\\
In fact also here we have ``Moebius cutting experiment''
$$
\C\BP^{1}\cdot\C-\C\BP^1=(\C+1)(\C-1))=\C^2-1=\R^4-1=S^3\R_+.
$$
showing that fibration $\BS^2(\C-1)$ is non trivial: $S^3\R_+$.

This remains true in general:
$$
\C\BP^{n-1}\cdot\C-\C\BP^{n-1}=\C\BP^{n-1}(\C-1)=\C^n-1=\R^{2n}-1=S^{2n-1}\cdot\R_+.
$$
The above may be repeated for the quaternions; we present the morphological headlines:\\
We have
\begin{gather*}
\BH=\R^4,\quad\BH^n=(\R^4)^n=\R^{4n},\\
\BH \BP^n=\frac{\BH^{n+1}-1}{\BH-1}=\BH^n+\BH^{n-1}+\cdots+1=\R^{4n}+\R^{4n-4}+\cdots+1\\
=\BH^{n}+\BH\BP^{n-1}.
\end{gather*}
Also
$$
\BH\BP^{n-1}=\frac{\R^{4n}-1}{\R^{4}-1}=\frac{S^{4n-1}}{S^{3}}
$$
leading to the $\BH\BP^n$-factorization (fibration)
$$
S^{4n-1}=\BH\BP^{n-1}\cdot S^3,
$$
which in particular for $n=2$ leads to the second Hopf-fibration
$$
S^7=\BH\BP^1\cdot S^3=(\BH+1)S^3=(\R^4+1)S^3=\BS^4S^3.
$$
The Moebius factorization is given by
$$
\BH\BP^n-1=\BH\BP^{n-1}\cdot\BH
$$
while we also have the Moebius cutting experiment:
$$
\BH\BP^{n-1}\cdot\BH-\BH\BP^{n-1}=\BH\BP^{n-1}\cdot(\BH-1)=\BH^n-1=\R^{4n}-1=S^{4n-1}\cdot\R_+.
$$
But not every interesting quotient in calculus leads to a morphological synthesis that produces a nice manifold. Yet these quotients are also interesting because they say a lot about the meaning of morphological calculus and we call them ``phantom geometrical objects''.
\begin{ex}$\R\BP_h^{2n}$\\
$$
\R\BP_h^{2n}=\frac{\R^{2n+1}+1}{\R+1}=\R^{2n}-\R^{2n-1}+\cdots+\R+1,
$$
which we call the Phantom (real) projective space of dimension $2n.$
\end{ex}
The simplest case is $$\R\BP_h^{2}=\R^2-\R+1,$$ with Euler characteristic $3$. This would be one too high for a connected $2-$manifold and $\R^2-\R+1$ corresponds to: take plane $\R^2,$ delete line $\R$ and add point $1$; it makes sense as a weird object but not as a 2-manifold.\\
In fact one could say
$$
\R^2-\R+1=(2\R_++1)\R-\R+1=2\R_+\R+1,
$$
two halfplanes (or half-discs or triangles) glued together by a single point (a butterfly).\\
Note that we also have that
$$
\R\BP_h^{2}=\frac{\BS^3}{\BS^1}.
$$

If we would now use $\R=2\R+1$ we could make the identification $\BS^3=S^2,$ $\BS^1=S^1$ and arrive at
$$
\frac{\BS^3}{\BS^1}=\frac{S^3}{S^1}=\BS^2=\R^2+1\;\; (Hopf\;fibration)
$$
and therefore
$$
\R^2-\R+1=\R^2+1.
$$
This is total nonsense because this identification is even wrong on the level of Euler numbers : $3=2.$\\

The reason why such bad identification happens is because the Euler numbers of $\BS^3, S^3, \BS^1, S^1$ are all equal to zero, so, on the level of Euler numbers:
$$
\frac{S^3}{S^1}=\frac{0}{0}\;\; \mathrm{\&}\;\;\frac{\BS^3}{\BS^1}=\frac{0}{0},
$$
so one would not even be allowed to consider the quotients $S^3/S^1,$ $\BS^3/\BS^1$. But that would also exclude $\C\BP^n$ from the picture as well as the Hopf fibration, an unpermitable exclusion. This is a sound reason why the relations $\R=2\R+1$ or $\R_+=\R$ or $\BS^n=S^n$ must be forbidden: they simply spoil the calculus.\\
The general Phantom projective space
$$
\R\BP_h^{2n}=\R^{2n}-\R^{2n-1}+\cdots-\R+1
$$
surely makes sense as a geometrical object, but the corresponding quantity $\R^{2n}-\R^{2n-1}+\cdots-\R+1$ still has negative numbers as coefficients, so it is not yet fully evaluated. This can be done by replacing $\R=2\R_+ + 1$ at suitable places, giving rise to
$$
\R^{2n}-\R^{2n-1}+\cdots-\R+1=2\R_+\R^{2n-1}+2\R_+\R^{2n-3}+\cdots+2\R_+\R+1,
$$
which also provides a synthesis for $\R\BP_h^{2n}$. Comparing $\R^{2n}-\R^{2n-1}+\cdots-\R+1$ with Poincar\'e polynomial also suggests that some of the homology spaces of $\R\BP_h^{2n}$ would have negative dimension. But we also have that phenomena with the object
$$
\R^2-1=(\R-1)(\R+1)=2\R_+\R+2\R_+.
$$
Quantity simply doesn't always have a positive evaluation as an addition of powers $\R^s$. This leads to
\begin{defn}
A morphological object is called integrable if it has an evaluation of the form
$$
\F=a_0\R^{n}+a_1\R^{n-1}+\cdots+a_n,\;a_0\in\N,\; a_1,\ldots,a_n\in\N\cup\{0\};
$$
this polynomial is then called the ``total quantity'' or ``integral''. The object $\F$ is called semi-integrable if it has an evaluation as an addition of terms of the form $\R^j_+\R^k.$ Such expression is not unique unless we require the power ``$j$'' of $\R_+$ to be minimal, in which case the obtained expression is also called ``total quantity'' or ``integral''.\\

Note that not every object is semi-integrable; for example $F=\R-2$ is an object and hence it has certain hidden quantity, but it cannot be evaluated as an addition in terms of $\R$ and $\R_+.$ One option would be to introduce new type of line, e.g.
$$
\R_+=2\R_{++} + 1,
$$
but that would not lead to be more interesting calculus.
\end{defn}
Notice that the Phantom projective plane can also be interpreted as the result of the cutting experiment
\begin{eqnarray*}
\R\BP_h^{2n}&=& \R^{2n}-\R^{2n-1}+\cdots-\R+1\\
           &=& \left(\R^{2n}+\R^{2n-2}+\cdots+1\right)-\left(\R^{2n-2}+\cdots+1\right)\R\\
           &=& \C\BP^{n}-\C\BP^{n-1}\cdot\R.
\end{eqnarray*}
We also have the Phantom Moebius strip
$$
\R\BP_h^{2n}-1=(\R-1)\C\BP^{n-1}\cdot\R
$$
and this time we have a Moebius ``glueing-experiment''
$$
(\R-1)\C\BP^{n-1}\cdot\R+(\R-1)\C\BP^{n-1}=\C\BP^{n-1}(\R^2-1)
$$
$$
=\R^{2n}-1=S^{2n-1}\cdot\R_+,
$$
the same cylinder as we had earlier on.\\
Of course one may also consider complex and quaternionic phantom projective spaces:
$$
\C\BP_h^{2n}=\frac{\C^{2n+1}+1}{\C+1}=\C^{2n}-\C^{2n-1}+\cdots-\C+1=\cdots,
$$
$$
\BH\BP_h^{2n}=\frac{\BH^{2n+1}+1}{\BH+1}=\BH^{2n}-\BH^{2n-1}+\cdots-\BH+1=\cdots
$$
The fact that the corresponding synthesis for phantom projective spaces does not add up to a
manifold implies that these quotients do not correspond to a group action (or else the quotients would be homogeneous spaces). Indeed, the denominators in the definition of the projective spaces are the multiplicative groups $\R-1, \C-1, \BH-1$ while for the phantom spaces we have the spheres $\BS^1=\R+1,$ $\BS^2=\C+1,$ $\BS^3=\BH+1$ which are non-groups leading to non-group actions. In fact group actions can not be recognised within morphological calculus itself, only by the outside interpretations. The consideration of phantom geometry also leads to the next definition.
\begin{defn}
A morphological object is said to be of ``integer type'' if it has an evaluation of the form
$$
\F=a_0\R^{n}+a_1\R^{n-1}+\cdots+a_n,\;a_0\in\N,\; a_1,\ldots,a_n\in\Z.
$$
A semi-integrable object that is not of integer type is to said to be of ``half integer type''. Other objects are ``just another type''.\\
\end{defn}
Notice that $\R-2$ is of integer type but not semi-integrable while the building blocks $\R_+^j\R^k,$ $j>0$ are semi integrable but not integer type: they are half-integer type.\\
The cylinder $(2\R+2)\R_+=\R^2-1$ is clearly of integer type but only semi-integrable while the small cylinder $(\R+1)\R_+$ is only of half-integer type. This example confirms that it is a good idea to keep two circles $S^1=2\R+2,$ $\BS^1=\R+1$ in use rather than deciding that $2\R+2=2(\R+1)$ is always a pair of circles. Note that the object $\R_+ - 1$ is just another type while $1-\R$ isn't even an object. So we have a kind of hierarchy that is quantity based.
\begin{ex}
Phantom fibrations\\
We already discussed Hopf factorization
$$
S^{2n-1}=(\R^n+1)S^{n-1}
$$
which only for $n=2$ and $n=4$ leads to a true fibration: the Hopf-fibrations
$$
S^3=\BS^2S^1,\;\;S^7=\BS^4S^3.
$$
\end{ex}
These fibrations in fact correspond to projective geometry and the factors $S^1$ and $S^3$ are group actions.\\
In the other cases like e.g.
$$
S^5=\BS^3S^2
$$
we don't have this. However also the product
$$
\BS^3S^2=(\R^3+1)S^2=\R^3S^2+S^2
$$
does lead to a synthesis of $S^5$ and it is like a fibration still, but an irregular fibration that would not locally correspond to a Cartesian product, whence the name ``phantom fibration''.\\
For the spheres $S^{2^n-1}$ we also have repeated factorizations
$$
S^7=\left(\R^4+1\right)S^3=\left(\R^4+1\right)\left(\R^2+1\right)\left(\R+1\right)2,
$$
$$
S^{15}=\left(\R^8+1\right)\left(\R^4+1\right)\left(\R^2+1\right)\left(\R+1\right)2,
$$
and so on. If we apply non-associativity we get
$$
S^7=\left(\left(\R^4+1\right)\left(\R^2+1\right)\right)S^1=\C\BP^3S^1,
$$
$$
S^7=\left(\R^4+1\right)\left(\left(\R^2+1\right)\left(\R+1\right)2\right)=\left(\R^4+1\right)S^3=\BS^4S^3,
$$
two fibrations of $S^7$ that follow from complex and quaternion geometry and that are unrelated.\\
There are also more Hopf factorizations, the simplest one being
$$
S^8=\left(\R^6+\R^3+1\right)S^2.
$$
In general they follow from products of the form
$$
\left(\R^{s\cdot k}+\R^{(s-1)\cdot k}+\cdots +\R^k+1\right)\left(2\R^{k-1}+2\R^{k-2}+\cdots +2\right)
$$
leading to
$$
S^{(s+1)k-1}=\left(\R^{s\cdot k}+\cdots +\R^k+1\right)S^{k-1}
$$
and they play a crucial role in the ``Gra\ss mann division problem''.\\
Needless to say that there are repeated factorizations of this type.\\
Also the addition formula for spheres may be generalised. \\
For $p+q+r=m$ we have
\begin{gather*}
\R^m-1=\left(\R^p-1\right)\left(\R^q-1\right)\left(\R^r-1\right)+\left(\R^p-1\right)\left(\R^q-1\right)\\
+\left(\R^p-1\right)\left(\R^r-1\right)+\left(\R^q-1\right)\left(\R^r-1\right)+\left(\R^p-1\right)+\left(\R^q-1\right)+\left(\R^r-1\right),
\end{gather*}
from which we obtain:
\begin{gather*}
S^{m-1}=S^{p-1}S^{q-1}S^{r-1}\R_+^2 + S^{p-1}S^{q-1}\R_+ + S^{p-1}S^{r-1}\R_+ + S^{q-1}S^{r-1}\R_+\\
 + S^{p-1} +S^{q-1} +S^{r-1}.
\end{gather*}
Needless to say also that our list of interesting manifolds and geometries is far from complete.

Let's take the Klein bottle as an example, we have the following morphological analysis. A Klein bottle can be obtained from a Moebius band by properly glueing a circle to the edge, thus closing it up into a compact $2$-manifold. As we know, a Moebius band may be obtained by removing a point from $\R\BP^2:\R\BP^2-1.$ Then one blow up the hole to a small disc and one glues a circle $S^1=2\R+2$ to that, giving $2$-manifold with boundary. Finally one identifies every point on this $S^1$ with its anti-podal point: $S^1/\Z_2$ which leads to a continuation across the boundary and to the Klein bottle. In morphological language we have:
\begin{gather*}
\left(\R\BP^2-1\right) + S^1/\Z_2 = \left(\R\BP^2-1\right)+\left(\R +1\right)\\
=\left(\left(\R^2+\R+1\right)-1\right)+\left(\R+1\right)\\
=\left(\R^2+\R\right)+\left(\R+1\right)\\
=\left(\R+1\right)\R+\left(\R+1\right)=\left(\R+1\right)\left(\R+1\right),
\end{gather*}
so we end up with a circle $\BS^1$-bundle over $\BS^1$. But $\BS^1\cdot \BS^1$ may also simply represent a torus: there is no way one can tell from the quantity $(\R+1)^2$ alone whether this represents a torus or a Klein-bottle. Only in the initial formula $\left(\R\BP^2-1\right)+\left(\R+1\right)$ one can specify a Klein bottle but as one starts calculating, this specification is lost.\\
Higher dimensional Klein bottles may be introduced as the ``blow up'' experiment:
\begin{gather*}
\left(\R\BP^n-1\right) + S^{n-1}/\Z_2 = \left(\R\BP^n-1\right)+\R\BP^{n-1}\\
=\R\BP^{n-1}\cdot\R+\R\BP^{n-1}=\R\BP^{n-1}\cdot\BS^1,
\end{gather*}
an $\BS^1$-bundle over $\R\BP^{n-1}$.\\
Similarly, complex and quaternionic Klein-bottles may be introduced as (exercise) the ``blow up experiment'':
\begin{gather*}
\left(\C\BP^n-1\right) + S^{2n-1}/S^1 = \C\BP^{n-1}\cdot\left(\C+1\right),\\
\left(\BH\BP^n-1\right) + S^{4n-1}/S^3 = \BH\BP^{n-1}\cdot\left(\BH+1\right).
\end{gather*}
To summarise this section, we notice that there is no one to one correspondence between morphological calculus and geometry. This may be seen as a drawback but it is also a stronghold because it means that there exists another perspective that reveals a hidden aspect of geometry: the quantity of an object.

\section{Groups and Homogeneous Spaces}
Groups enter morphological calculus via a proper morphological analysis; the group structure will be lost and the organised quantity remains.\\
We begin with the groups $O(n)$, $SO(n)$, $GL(n,\R)$, $GL(n,\R)$, $SL(n,\R)$.\\
The orthogonal group $O(n)$ is the group of all orthogonal matrices $(a_{ij})$. If we represent such a matrix as a row $(\b{a}_1,\ldots,\b{a}_n)$ of column vectors it simply means that $\b{a}_1,\ldots,\b{a}_n$ are orthogonal unit vectors. This means that one can start off by choosing
$$
\b{a}_1\in S^{n-1}
$$
followed by choosing
$$
\b{a}_2\in S^{n-1}\cap \{\lambda\b{a}_1,\;\lambda \in\R\}^\bot=S^{n-2}
$$
and then
$$
\b{a}_3\in S^{n-1}\cap \{\lambda_1\b{a}_1+\lambda_2\b{a}_2,\;\lambda_j \in\R\}^\bot=S^{n-3}
$$
and so on, until for $\b{a}_n$ there are just $2$ choices
$$
\b{a}_n\in S^{n-1}\cap \mathrm{span}\{\b{a}_1,\ldots,\b{a}_{n-1}\}^\bot=S^{0}.
$$
This immediately leads to the morphological definition
$$
O(n)=S^{n-1}\cdot S^{n-2}\cdots S^0,
$$
as well as to the recursion formula
$$
O(n)=S^{n-1}\cdot O(n-1),\;\; O(0)=1.
$$
For the group $SO(n)$ everything remains the same except that for the last vector $\b{a}_n$ there is just one choice, determined by $\mathrm{det}(a_{ij})=1$ condition.\\
We thus have the definition
$$
SO(n)=S^{n-1}S^{n-2}\cdots S^{1}=O(n)/\Z_2.
$$
Clearly $O(n), SO(n)$ are integrable and the integral is obtained by substituting $S^{j-1}=2\R^{j-1}+\cdots+2$  and working out the product.\\
The general linear group $GL(n,\R)$ is obtained similarly by writing matrix $(a_{ij})$ as $(\b{a}_1,\ldots,\b{a}_n)$ whereby
\begin{eqnarray*}
\b{a}_1 &\in & \R^n\setminus\{0\},\\
\b{a}_2 &\in & \R^n\setminus\mathrm{span}\{\b{a}_1\},\\
 \vdots & & \\
 \b{a}_n &\in & \R^n\setminus\mathrm{span}\{\b{a}_1,\ldots,\b{a}_{n-1}\}
\end{eqnarray*}
which leads to the morphological definition
$$
GL(n,\R)=\left(\R^n-1\right)\left(\R^n-\R\right)\cdots\left(\R^n-\R^{n-1}\right).
$$
We readily obtain the quotient formula
\begin{eqnarray*}
\frac{GL(n,\R)}{O(n)} & = & \frac{\left(\R^n-1\right)\left(\R^n-\R\right)\cdots\left(\R^n-\R^{n-1}\right)}{S^{n-1}\cdot S^{n-2}\cdots S^{0}}\\
                      & = & \left(\frac{\R^n-1}{S^{n-1}}\right)\left(\frac{\R^n-\R}{S^{n-2}}\right)\cdots\frac{\left(\R^n-\R^{n-1}\right)}{S^0}\\
                      & = & \R_+\cdot \left(\R\cdot\R_+\right)\cdots\left(\R^{n-1}\cdot\R_+\right),
\end{eqnarray*}
which symbolizes the GRAMM-SCHMIDT orthogonalization procedure. Here we applied commutativity of the product but that doesn't matter too much; in fact one can also write
$$
GL(n,\R)=\left(S^{n-1}\R_+\right)\left(\R S^{n-2}\R_+\right)\cdots\left(\R^{n-1}S^0\R_+\right).
$$
For the group $SL(n,\R)$ we have the extra condition $\mathrm{det}\left(a_{ij}\right)=1,$ which readily leads to
$$
SL(n,\R) = \frac{GL(n,\R)}{\R-1},
$$
and so also
$$
\frac{SL(n,\R)}{SO(n)}=\R_+\cdot \left(\R\cdot\R_+\right)\cdots\left(\R^{n-2}\cdot\R_+\right)\R^{n-1}.
$$
Now let us look some homogeneous spaces.\\

The Stiefel manifold $V_{n,k}(\R)$ is by definition the manifold of orthonormal $k$-frames $(\b{v}_1,\ldots, \b{v}_k)$ in $\R^n$. We hence have that for $k<n$:
$$
V_{n,k}(\R)=\frac{SO(n)}{SO(n-k)} = S^{n-1}\cdots S^{n-k}= \frac{O(n)}{O(n-k)}.
$$

The Stiefel manifold $V_{n,k}(\R)$ is the manifold of $k$-frames $(\b{v}_1,\ldots, \b{v}_k)$ that are linearly independent and hence span a $k$-plane. We have for $k<n$:
$$
VL_{n,k}(\R)=\frac{GL(n,\R)}{\R^{n-k}\cdot GL(n-k,\R)} = \left(\R^{n}-1\right)\left(\R^{n}-\R\right)\cdots \left(\R^{n}-\R^{k-1}\right).
$$
The Gra\ss mann manifold $G_{n,k}(\R)$ is the manifold of $k$-dimensional subspaces of $\R^n.$ Now, each $k$-dimensional subspace has an orthogonal frame and that can be chosen in $O(k)$ in different ways. This leads to the combinatorial formula:
$$
G_{n,k}(\R)=\frac{VL_{n,k}(\R)}{O(k)} = \frac{O(n)}{O(k)\cdot O(n-k)}=\frac{S^{n-1}\cdot S^{n-2}\cdots S^{n-k}}{S^{k-1}\cdots S^{0}}.
$$
The Gra\ss mann manifold may also be constructed starting from the general linear group:
$$
G_{n,k}(\R)=\frac{V_{n,k}(\R)}{GL(k,\R)} =  \frac{\left(\R^{n}-1\right)\left(\R^{n}-\R\right)\cdots \left(\R^{n}-\R^{k-1}\right)}{\left(\R^{k}-1\right)\cdots \left(\R^{k}-\R^{k-1}\right)}
$$
and the equivalence of both definitions readily follows from the Gramm-Schmidt factorization.\\

By $\widetilde{G_{n,k}}(\R)$ we denote the manifold of all ORIENTED $k$-dimensional subspaces of $\R^n,$ i.e.,
$$
\widetilde{G_{n,k}}(\R) = \frac{V_{n,k}(\R)}{SO(k)} = \frac{SO(n)}{SO(k)\cdot SO(n-k)} = \frac{S^{n-1}\cdot S^{n-2}\cdots S^{n-k}}{S^{k-1}\cdots S^{1}}.
$$
Now, for the Stiefel manifolds everything is clear, but for the Gra\ss mann manifolds we have one major problem.
\begin{problem}
Gra\ss mann division problem\\
Can one work out the polynomial division $\frac{S^{n-1}\cdot S^{n-2}\cdots S^{n-k}}{S^{k-1}\cdots S^{0}}$, and does it result in an integral (polynomial in ``$\R$'' with natural number coefficients).
\end{problem}
To solve the problem we will work with the quotient $VL_{n,k}(\R)/GL(k,\R)$ that is equivalent and easier to work with. For the case of simplicity take $k=3$. Every $3D$-subspace $V$ of $\R^n$ is spanned by $3$ linearly independent vectors $\b{v}_1, \b{v}_2, \b{v}_3$ that may be chosen in $GL(3,\R)$ different ways. For each $V$ there is a unique triple $(\b{v}_1, \b{v}_2, \b{v}_3)$ that may be written as a matrix of the form
\[
\begin{bmatrix}
  \b{v}_1 \\
  \b{v}_2 \\
  \b{v}_3
 \end{bmatrix}\!\!=\!\!
\setcounter{MaxMatrixCols}{20}
 \begin{bmatrix}
  c_{11} & \!\!\cdots\!\! & c_{1j_3} & 0 & c_{1j_3+2} & \!\!\cdots\!\!  & c_{1j_2}  & 0  &  c_{1{j_2}+2} & \cdots & c_{1j_1} & 1  &  0 & \!\!\cdots\!\! & 0\\
  c_{21} & \!\!\cdots\!\! & c_{2j_3} & 0 & c_{2j_3+2} & \!\!\cdots\!\!   & c_{2j_2} & 1  &  0            & \cdots & 0        & 0  &  0 & \!\!\cdots\!\! & 0\\
 c_{31}  & \!\!\cdots\!\! & c_{3j_3} & 1 &     0        & \!\!\cdots\!\!   & 0        & 0  &  0            & \cdots & 0        & 0  &  0 & \!\!\cdots\!\! & 0
 \end{bmatrix}
 \]
 and any other frame in $V$ may be obtained  by a unique $GL(3,\R)$-action from this, so in fact the division is carried out by looking to matrices of the above special form. As the coefficients $c_{ij}$ vary the matrices of the above form constitute a cell of $G_{n,3}(\R)$ that is a copy of a certain $\R^j$ and it is called a Schubert cell. We thus have proved the following results.
\begin{thm}Schubert cells\\
The object $G_{n,k}(\R)=\R^d+c_1\R^{d-1}+\cdots+c_d$ whereby $c:j\in\N$ is the number of Schubert cells of dimension $d-j.$
\end{thm}
Apart from this there are typical morphological questions such as:
\begin{description}
\item Q1: To decompose $G_{n,k}(\R)=O_1\cdots O_s$ as a (e.g. maximal) product of morphological objects $O_j$ of integer type (that are irreducible e.g.).
\item Q2: To look for $G_{n,k}(\R)$-factorization $O_1\cdots O_t$ in terms of objects $O_j$ that are integrable.
\end{description}
Let us consider a few examples of such Gra\ss mann factorizations.\\
Of course we readily have $G_{n,k}(\R)=\R\BP^{n-1}$ and the Hopf factorizations provide further ways of factorizing this.\\
Next for $k=2$ we have:
\begin{eqnarray*}
G_{2n,2}(\R)    & = & \frac{S^{2n-1}\cdot S^{2n-2}}{S^1\cdot S^0} = \C\BP^{n-1}\cdot\R\BP^{2n-2}\\
G_{2n+1,2}(\R)  & = & \frac{S^{2n}\cdot S^{2n-1}}{S^0\cdot S^1} = \R\BP^{2n}\cdot\C\BP^{n-1},
\end{eqnarray*}
showing a clear $2-$periodicity.\\

For $k=3$ the first interesting case is
$$
G_{6,3}(\R)=\frac{S^5\cdot S^4\cdot S^3}{S^2\cdot S^1\cdot S^0},
$$
which, using the Hoft factorizations
$$
S^5=\left(\R^{3}+1\right)S^2 = \BS^3S^2, \;\;S^3 = \left(\R^{2}+1\right)S^1 = \BS^2S^1
$$
may be evaluated as
$$
G_{6,3}(\R) = \left(\R^{3}+1\right)\cdot\R\BP^4\cdot\left(\R^{2}+1\right) = \R\BP^4\cdot \BS^3\cdot\BS^2.
$$
Note here that it is forbidden to divide $\BS^2/S^2=1.$\\
More interesting still is the next case
$$
G_{7,3}(\R)=\frac{S^6\cdot S^5\cdot S^4}{S^2\cdot S^1\cdot S^0},
$$
which, using the Hopf-factorization $S^5=(\R^3+1)S^2$ yields.
$$
G_{7,3}(\R) = \frac{\R\BP^6\cdot\left(\R^{3}+1\right)\R\BP^4}{(\R+1)}.
$$
Now $\R\BP^4$ and $\R\BP^6$ cannot be divided by $(\R+1)$; in fact these objects are irreducible in morphological sense. So, the division that works here is:
$$
\R\BP_h^2 = \frac{\R^3+1}{\R+1}=\R^2-\R+1,
$$
the phantom projective plane, leading to the following maximal factorization
$$
G_{7,3}(\R) = \R\BP^6\cdot\R\BP^4\cdot\R\BP^2_h
$$
in terms of irreducible objects of integer type.\\
But now the factors are no longer integrable, which also shows that the integrability of Gra\ss mann manifolds is in fact not so trivial. But we have:
\begin{eqnarray*}
\frac{\left(\R^3+1\right)}{\R+1}\R\BP^4 &=& \frac{\left(\R^3+1\right)}{\R+1}\left(\R^2\frac{\left(\R^3-1\right)}{\R-1}+(\R+1)\right)\\
                                       &=& \R^2\frac{\left(\R^6-1\right)}{\R^2-1}+\left(\R^3+1\right)  = \C\BP^2\cdot\R^2+\BS^3,
\end{eqnarray*}
so that in fact we have integrable factorization
$$
G_{7,3}=\R\BP^6\cdot\left(\C\BP^2\cdot\R^2+\BS^3\right).
$$
The next case is again simpler:
$$
G_{8,3}(\R)=\frac{S^7\cdot S^6\cdot S^5}{S^2\cdot S^1\cdot S^0} = \left(\C\BP^3\cdot\R\BP^6\right)\left(\R^3+1\right).
$$
For the next case
$$
G_{9,3}(\R)=\frac{S^8\cdot S^7\cdot S^6}{S^2\cdot S^1\cdot S^0},
$$
we have to use the next Hopf factorization
$$
S^8=\left(\R^6+\R^3+1\right)S^2,
$$
which gives us:
$$
G_{9,3}(\R)=\left(\R^6+\R^3+1\right)\C\BP^3\cdot\R\BP^6.
$$
The next cases are:
$$
G_{10,3}(\R)=\frac{S^9\cdot S^8\cdot S^7}{S^2\cdot S^1\cdot S^0} = \C\BP^4\cdot\left(\R^6+\R^3+1\right)\cdot\R\BP^7,
$$
the first appearance of an odd dimensional $\R\BP^n,$ and
$$
G_{11,3}(\R)=\frac{S^{10}\cdot S^9\cdot S^8}{S^2\cdot S^1\cdot S^0} = \R\BP^{10}\cdot\C\BP^4\cdot\left(\R^6+\R^3+1\right).
$$
In the next case we again have $2$ odd spheres and the Hopf factorization
$$
S^{11} = \left(\R^6+1\right)\left(\R^3+1\right)S^2 = \BS^6\cdot\BS^3\cdot S^2,
$$
giving rise to
$$
G_{12,3}(\R)=\frac{S^{11}\cdot S^{10}\cdot S^9}{S^2\cdot S^1\cdot S^0} = \R\BP^{10}\cdot\C\BP^4\cdot\BS^6\cdot\BS^3.
$$
and finally in the next case we again have two irreducible spheres $S^{12}, S^{10}$, leading to
$$
G_{13,3}(\R)=\frac{S^{12}\cdot S^{11}\cdot S^{10}}{S^2\cdot S^1\cdot S^0} = \R\BP^{12}\cdot\R\BP^{10}\cdot\BS^6\frac{\left(\R^3+1\right)}{\R+1}.
$$
where once again, the phantom projective plane appears
$$
\frac{\R^3+1}{\R+1} = \R^2 - \R + 1 = \R\BP^2_h.
$$
There is clearly a $6$-periodicity in the factorization of Gra\ss mann manifolds for $k=3.$ The formulas obtained here lead to a classification but they do not correspond to the fibre bundles of any kind. Besides, we used repeatedly the fact that quantity is commutative. Another interesting homogeneous space is Flag Manifold $F_{n; k, \ell}(\R)$, $k<\ell<n$ whereby $W$ is subspace of $\R^n$ of dimension $1$ and $V\subset W$ is a subspace of dimension $k$. This clearly leads to the fibration
\begin{eqnarray*}
F_{n; k, \ell}(\R) &=& G_{n, \ell}(\R)\cdot G_{\ell, k}(\R)\\
                   &=& \frac{O(n)}{O(n-\ell)O(\ell)}\cdot\frac{O(\ell)}{O(\ell-k)\cdot O(k)}=\frac{O(n)}{O(k)O(\ell-k)\cdot O(n-\ell)}.
\end{eqnarray*}
The flag manifold $F_{n; k, \ell}(\R)$ may also be seen as manifold $(V,V')$ with $V\subset\R^n$ a subspace of dimension $k$ and $V\perp V'$ of dimension $\ell-k.$ The link with the previous definition simply follows from $W=V\oplus V'$ and we have the fibration
\begin{eqnarray*}
F_{n; k, \ell}(\R) &=& G_{n, k}(\R)\cdot G_{n-k,\ell-k}(\R)\\
                   &=& \frac{O(n)}{O(k)O(n-k)}\cdot\frac{O(n-k)}{O(\ell-k)\cdot O(n-\ell)}=\frac{O(n)}{O(k)O(\ell-k)\cdot O(n-\ell)}.
\end{eqnarray*}
More in general for $0<k_1<\ldots k_s<n$ we may define the flag manifold $F_{n; k_1,\ldots,k_s}(\R)$ as the manifold of flags $(V_1\ldots,V_s)$ with $V_1\subset V_2\cdots\subset V_s\subset\R^n$ subspaces of dimension $\mathrm{dim}V_j=k_j.$ We clearly have the iterated fibration
\begin{align*}
F_{n; k_1,\ldots,k_s}(\R) = G_{n, k_s}(\R)\cdot G_{k_s,k_{s-1}}(\R)\cdots G_{k_2,k_1}(\R)&&\\
                  =\frac{O(n)}{O(n-k_s)O(k_s)}\cdot\frac{O(k_s)}{O(k_s-k_{s-1})\cdot O(k_{s-1})}\cdots\frac{O(k_2)}{O(k_2-k_1)\cdot O(k_1)}&& \\
                  =\frac{O(n)}{O(n-k_s)O(k_s-k_{s-1})\cdots O(k_2-k_1)O(k_1)}&&.
\end{align*}
Using orthogonal subspaces, we have:
$$
F_{n; k_1,\ldots,k_s}(\R) = G_{n, k_1}(\R)\cdot G_{n-k_1,k_2-k_1}(\R)\cdots G_{n-k_{s-1},k_s-k_{s-1}}(\R).
$$
Orthogonal groups may also be defined for the spaces $\R^{p,q}$ with pseudo-Euclidean inner product
$$
\langle x,y\rangle=x_1y_1 + \cdots + x_py_q - x_{p+1}y_{p+1} - \cdots - x_{p+q}y_{p+q}.
$$
The corresponding groups are $O(p,q)$ and $SO(p,q)$. The group $SO(p,q)$ e.g. is determined as the manifold of frames of signature $(p,q):$
$$
\left(\b{v}_1,\ldots, \b{v}_p; \b{v}_{p+1},\ldots,\b{v}_{p+q}\right)
$$
whereby $\b{v}_1\in S^{p-1}\cdot\R^q$ is the first spacelike vector $\b{v}_2\perp\b{v}_1\in S^{p-2}\cdot\R^q$
up to $\b{v}_p\perp\mathrm{span}\left(\b{v}_1,\ldots,\b{v}_{p-1}\right)\in S^{0}\cdot\R^q$
and the remaining vectors $\left(\b{v}_{p+1},\ldots,\b{v}_{p+q}\right)$ form a right oriented time-like $q$-frame, i.e. $\b{v}_{p+1}\in S^{q-1}$, $\b{v}_{p+2}\in S^{q-2}$ and $\b{v}_{p+q}$ is fixed by the fact that the determinant of the whole frame equals $+1.$

In total, the morphological bill adds up to:
\begin{eqnarray*}
SO(p,q) &=& \left(S^{p-1}\cdot\R^q\right)\cdots\left(S^0\cdot\R^q\right)S^{q-1}\cdots S^1\\
                   &=& O(p)\cdot SO(q)\cdot\R^{p\cdot q}.
\end{eqnarray*}
and it is a two component group still.

All of the above may be generalized to the complex Hermitian case. Let us start with $\C^n$ provided with the Hermitian inner product:
$$
(\b{z},\b{w})=z_1\overline{w}_1+ \cdots + z_n\overline{w}_n.
$$
Then by $U(n)$ we denote the unitary group of matrices learning the Hermitian form invariant; its matrices may be written as Hermitian orthonormal frames $\b{v}_1,\ldots,\b{v}_n$ whereby $|\b{v}_j|=1,$ $(\b{v}_j,\b{v}_k)=0$ for $j\neq k$.\\
This leads to the following morphological analysis:\newline
$\b{v}_1\in S^{2n-1}$ is the unit vector in $\C^n=\R^{2n}$,\newline
$\b{v}_2\perp\b{v}_1$ in hermitian sense, i.e. $\b{v}_2\in\b{v}_1^\perp\cap S^{2n-1}=S^{2n-3}$\newline up to\\
$\b{v}_n\perp\b{v}_1,\ldots,\b{v}_{n-1}$, i.e. $\b{v}_n\in S^1$\newline
and, therefore,\newline
$$U(n)=S^{2n-1}\cdot S^{2n-3}\cdots S^{1}.$$
In the above, please note that $\langle\b{v},\b{w}\rangle=\mathrm{Re}(\b{v},\b{w})$ is the orthogonal inner product in $\R^{2n}$ and so \newline
$(\b{v},\b{w})=0$ iff $\langle\b{v},\b{w}\rangle=0$ and $\langle i\b{v},\b{w}\rangle=0.$

Clearly $U(n)$is a subgroup of $SO(2n)$ and for the quotient we have:
$$
\frac{SO(2n)}{U(n)}=S^{2n-2}\dot S^{2n-4}\cdots S^{2},
$$
which actually is a manifold, namely the manifold of all complex structures on $\mathbb{R}^{2n}$ (Exercise).

The special unitary group $SU(n)$ is the subgroup of matrices in $U(n)$ with determinant $=1$, i.e.,
$$
SU(n)=S^{2n-1}\cdots S^{3},
$$
and in particular $SU(2)=S^3.$

The definition of the complex general and special linear groups is obvious; they are denoted by $GL(n,\mathbb{C})$, $SL(n,\mathbb{C})$. Like for the orthogonal groups also for the complex group $U(n)$ we have the associated homogeneous spaces, in particular Gra\ss mann manifolds
$$
G_{n,k}(\mathbb{C})=\frac{U(n)}{U(k)\cdot U(n-k)}=\frac{S^{2n-1}\cdots S^{2n-2k+1}}{S^{2k-1}\cdot S^{2k-3}\cdots S^{1}},
$$
so for example
$$
G_{4,2}(\mathbb{C})=\frac{S^{7}\cdot S^{5}}{S^{3}\cdot S^{1}}=(\mathbb{R}^4+1)\mathbb{C}\mathbb{P}^2=\mathbb{S}^4\mathbb{C}\mathbb{P}^2=\mathbb{H}\mathbb{P}^1\cdot \mathbb{C}\mathbb{P}^2.
$$

$$
G_{5,2}(\mathbb{C})=\frac{S^{9}\cdot S^7}{S^{3}\cdot S^{1}}=\mathbb{C}\mathbb{P}^4\cdot(\mathbb{R}^4+1)=\mathbb{C}\mathbb{P}^4\cdot\mathbb{H}\mathbb{P}^1.
$$

$$
G_{6,2}(\mathbb{C})=\frac{S^{11}\cdot S^{9}}{S^{3}\cdot S^{1}}=\mathbb{H}\mathbb{P}^2\cdot\mathbb{C}\mathbb{P}^4.
$$
and so we have again a clear $2$-periodicity.

We leave the discussion of $G_{n,3}(\mathbb{C})$ as an exercise.

Unitary groups may also be constructed over spaces $\C^{p,q}$ with pseudo-Hermitian form
$$
(\underline{z},\underline{w})=z_1\overline{w}_1+\cdots+z_p\overline{w}_p - z_{p+1}\overline{w}_{p+1}-\cdots - z_{p+q}\overline{w}_{p+q}
$$
and the corresponding invariance groups are denoted by $U(p,q)$ and $SU(p,q)$ in case $\det =1$.

The corresponding frames now have to be chosen on the pseudo-Hermitian unit sphere:

$$
|z_1|^2+\cdots + |z_p|^2 - |z_{p+1}|^2 - \cdots -|z_{p+q}|^2 =1
$$
which leads to the morphological formula
$$
U(p,q)=(S^{2p-1}\cdot\mathbb{C}^q)\cdot(S^{2p-3}\cdot\mathbb{C}^q)\cdots(S^{1}\cdot\mathbb{C}^q)\cdot S^{2q-1}\cdot S^{2q-3}\cdots S^1.
$$
Of course we also have the complexified versions $O(n,\C)$ and $SO(n,\C)$ of $O(n)$ and $SO(n)$; it is another story which we'll leave out for the moment.

To finish the list of matrix groups leading to morphological analysis, we mention the compact symplectic groups $Sp(n)$; they follow from the quaternionic Hermitian form
$$
(q,\underline{r})=q_1\overline{r}_1+\cdots+q_n\overline{r}_n,
$$
whereby $q_j=q_{j_0}+iq_{j_1}+jq_{j_2}+kq_{j_3}$ is a quaternion and $\overline{q}_j=q_{j_0}-iq_{j_1}-jq_{j_2}-kq_{j_3}$ its quaternion conjugate.

$Sp(n)$ is by definition the goup of quaternion $n\times n$ matrices leaving this form invariant and its matrix elements may be regarded as quaternionic frames $q_1,\ldots,q_n$ whereby $q_r\in\mathbb{H}^n$ with $|q_1|=1$, i.e., $q_1\in S^{4n-1},$ $q_2\in\mathbb{H}^n$ with $|q_2|=1$ and $(q_1,q_2)=0$, i.e., $q_2\in S^{4n-5},$ and so on. This leads to the morphological bill:
$$
Sp(n)=S^{4n-1}\cdot S^{4n-5}\cdot S^{3},
$$
in particular $Sp(1)=S^3$ and $Sp(2)=S^7\cdot S^3.$

Also here may be investigated quaternionic Gra\ss mannians.

The groups $Sp(n)$ should not be confused with the non-compact groups $Sp(2n,\R)$ of matrices $A\in GL(2n,\R)$ leaving the maximal $2-$form invariant.

For $Sp(2n,\R)$ we did not find a morphological evaluation yet.

To finish this section we discuss the Spin groups $Spin(m)$.

We start by considering the real $2^m-$dimensional Clifford algebra $\R_m$ with generators $e_1,\ldots,e_m$ and relations $e_j\,e_k+e_k\,e_j=-2\delta_{jk}.$

The space of bivectors
$$
\R_{m,2}=\left\{\sum_{i,j}b_{ij}e_i\,e_j:\;b_{ij}\in \R\right\}
$$
forms a Lie algebra for the commutation product and the corresponding group is the Spin group:
$$
Spin(m)=exp(\R_{m,2}).
$$
Its elements may be written into the form $s=\underline{w}_1\cdots\underline{w}_{2s}$,
$\underline{w}_j=\sum w_{jk}e_k\in \R^m$ with $\underline{w}_j^2=-1$, i.e., $\underline{w}_j\in S^{m-1}$.

We have the following $Spin(m)$ representation
$$
h:Spin(m)\rightarrow SO(m)
$$
whereby
$$
h:s\rightarrow h(s):\; \underline{x}\rightarrow s \underline{x} \overline{s}
$$
whereby for $a\in\R_m$, $\overline{a}$ is the conjugation with properties $\overline{ab}=\overline{b}\,\overline{a}\; \& \;\overline{e_j}=-e_j.$

In this way $Spin(m)$ is a $2-$fold covering group of $SO(m)$, i.e.,
$$
SO(m)=Spin(m)/_{\mathbb{Z}_2}
$$
and also $Spin(m)$ is simply connected.

This might suggest the morphological evaluation
$$
Spin(m)=SO(m)\cdot \mathbb{Z}_2=S^{m-1}\cdots S^{1}\cdot S^{0}=O(m),
$$
which, through not entirely wrong in the sense of quantity, is somewhat uninteresting.

But there is a more interesting evaluation of $Spin(m)$.

Let us start with
\begin{align*}
\mathrm{Spin(3)}=&\{q_0 + q_1e_{23} + q_2e_{31} + q_3e_{12}\,:\,q\overline{q}=1\}\\
=& S^3=\mathbb{S}^2\cdot S^1=SU(2)=Sp(1)
\end{align*}
with differs rather substantially from
$$
O(3)=S^{2}\cdot S^{1}\cdot S^{0}.
$$

So in fact, the rotation group $SO(3)$ has two different representations in morphological calculus:

one as the matrix group
$$
\mathbb{S}O(3)=S^{2}\cdot S^{1}
$$
and one in terms of the Spin group (quaternion $S^3$):
\begin{align*}
\mathbb{S}O(3)=&\mathrm{Spin}(3)/_{\mathbb{Z}_2}=S^3/_{2}\\
=&\mathbb{S}^2 S^1/_{2}=\mathbb{S}^2\mathbb{S}^1=(\R^2+1)(\R+1)=\R\mathbb{P}^3.
\end{align*}

In general we got
$$
\mathbb{S}O(m)=\mathrm{Spin}(m)/_{2},
$$
which is another morphological version of the rotation group.

For $m=4$ we consider the pseudoscalar $e_{1234}$ with $e_{1234}^2=+1$ and $e_{1234}$ is central in the even subalgebra
$$
\R^{+}_4=\mathrm{Alg}\{e_{jk}\,:\,j<k\}\cong\R_3\cong\mathbb{H}\oplus\mathbb{H}.
$$
Putting
$$
E_{\pm}=\frac{1}{2}\left(1\pm e_{1234}\right)
$$
we have
$$
E_{+}+E_{-}=1,\;\;E_{\pm}^2=E_{\pm},\;\;E_{+}E_{-}=0,
$$
so every $a\in\R^{+}_4$ may be written uniquely as:
$$
a=a_+E_++a_-E_-,\;\;a_{\pm}\in\mathbb{H}=\mathrm{span}\{1,e_{23},e_{31},e_{12}\}
$$
and in particular
$$
s\in \mathrm{Spin}(4):s_+E_++s_-E_-,\; s_{\pm}\in S^3.
$$
So we have the morphological analysis

\begin{align*}
\mathrm{Spin}(4)=&\mathrm{Spin}(3)\times\mathrm{Spin}(3)\\
=&S^3\cdot S^3=S^3\cdot \mathbb{S}^2\cdot S^1.
\end{align*}
For $m=5$, we use the fact that
\begin{align*}
\mathrm{Spin}(5)=&\{s\in \R^{+}_5\,:\,s\overline{s}=1\},
\end{align*}
together with the isomorphisms
$$
\R^{+}_5=\R_4\cong\mathbb{H}(2)
$$
i.e., the set of $2\times 2$ quaternions matrices
$$
a=\begin{pmatrix}
a_{11} & a_{12} \\
a_{21} & a_{22}
\end{pmatrix},\;a_{ij}\in\mathbb{H}
$$
and under this isormorphism we also have
$$
\overline{a} =\begin{pmatrix}
\overline{a}_{11} & \overline{a}_{21} \\
\overline{a}_{12} & \overline{a}_{22}
\end{pmatrix}.
$$
This shows that in fact
$$
\mathrm{Spin}(5)=\mathrm{Sp}(2)=S^7\cdot S^3=\mathbb{S}^4\cdot S^3\cdot S^3=\mathbb{S}^4\cdot S^3\cdot \mathbb{S}^2\cdot S^1.
$$

For $m=6$, the pseudoscalar $e_{123456}$ satisfies $e_{123456}^2=-1$ and it is central in even subalgebra $\R_6^+=\R_5$, so it may identified with complex number $i$, leading to
$$
\R_6^+\cong \mathbb{C}\otimes\R_5^+\cong\mathbb{C}\otimes\mathbb{H}(2)\cong\mathbb{C}(4),
$$
and under this map $\R_6^+\rightarrow \mathbb{C}(4)$, the conjugate $\overline{a}$ of $a\in\R_6^+$ corresponds to the Hermitian conjugate $(a)^+$ of matrix $(a)\in \mathbb{C}(4)$.

Hence the group $G=\{a:\,a\overline{a}=1,\; a\in\R_6^+\}$ corresponds to $U(4).$ But for $m>5$, the group $G$ no longer corresponds to $\mathrm{Spin}(m)$ and for $m=6$
\begin{align*}
G =& \exp\left\{\sum b_{ij}e_{ij}+e_{123456}\right\}\\
  =& \exp\left\{\sum b_{ij}e_{ij}\right\}\times \exp\left\{e_{123456}\right\}=\mathrm{Spin}(6)\times U(1)
\end{align*}
which shows that really
$$
\mathrm{Spin}(6)=SU(4)= S^7\cdot S^5\cdot S^3 = S^5\cdot \mathbb{S}^4\cdot S^3\cdot \mathbb{S}^2\cdot S^1.
$$
For $m=7$ on the situation is much more complicated. Could it be that
$$
\mathrm{Spin}(7) = \mathbb{S}^6\cdot S^5\cdot \mathbb{S}^4\cdot S^3\cdot \mathbb{S}^2\cdot S^1?
$$
\section{Nullcones and Things}
The nullcone $NC^{n-1}$ of complex dimension $n-1$ in the locus of points $(z_1,\ldots,z_n)\in\C^n$ that satisfy $z_1^2+\ldots+z_n^2=0.$\\
The complex $(n-1)$-sphere $\C S^{n-1}$ consists of the solutions $(z_1,\ldots,z_n)$ of the equation $z_1^2+\ldots+z_n^2=1$.\\
It is non-compact manifold that admits a canonical compactification $\overline{\C S}^{n-1}\subset\C\BP^n$ given by the equation in homogeneous coordinates $z_1,\ldots, z_{n+1}:$
$$
z_1^2+\cdots+z_n^2=z_{n+1}^2
$$
that is equivalent to $z_1^2+\cdots+z_n^2+z_{n+1}^2=0$ if we replace $z_{n+1}\rightarrow \mathbf{i}z_{n+1}.$ The submanifold $\C S^{n-1}$ corresponds to the intersection with the region $z_{n+1}\neq0$ while the ``points at infinity'' corresponds to the intersection with plane $z_{n+1}=0,$ leading to:
$$
\overline{\C S}^{n-2}\;:\;z_1^2+\cdots+z_n^2=0.
$$
Hence we have the disjoint union
$$
\overline{\C S}^{n-1}=\C S^{n-1}\cup\overline{\C S}^{n-2}.
$$
We are going to perform the morphological calculus of those objects in two different ways, leading to two different formulas for the quantity (once again). The first method could be called the real geometry approach.\\
Let us write $z=\underline{x}+\mathbf{i}\underline{y}$, $\underline{x}=(x_1,\ldots,x_n)$, $\underline{y}=(y_1,\ldots,y_n)\in\R^n;$\\
then the equation for $NC^{n-1}$ may be rewritten as
$$
|\underline{x}|^2 = |\underline{y}|^2\;\;\&\;\;\langle\underline{x},\underline{y}\rangle=0
$$
with $|\underline{x}|^2 = x_1^2+\cdots+x_n^2,$ $\langle\underline{x},\underline{y}\rangle=x_1y_1+\cdots+x_ny_n.$ \\
First solution is the point $\underline{z} = 0$ with quantity $1$.\\
For $\underline{z} \neq 0$ we may write $\underline{x}=\rho\underline{\omega}$, $\underline{y}=\rho\underline{\nu}$, $\rho\in\R_+$ and $\underline{\omega}, \underline{\nu}\in S^{n-1}$ such that $\underline{\omega}\perp\underline{\nu}$, i.e., $(\underline{\omega},\underline{\nu})\in V_{n,2}(\R)$.
Hence we have
$$
NC^{n-1} = 1 + V_{n,2}(\R)\cdot\R_+ = 1 + S^{n-1}\cdot S^{n-2}\cdot\R_+.
$$
The complex sphere $\C S^{n-1}$ written in real coordinates would lead to:
$$
|\underline{x}|^2 = 1 + |\underline{y}|^2, \;\; \langle\underline{x},\underline{y}\rangle=0.
$$
First we have the case $|\underline{y}| = 0$, $|\underline{x}|=1$ leading to the quantity $S^{n-1}.$ Next for $|\underline{y}|\in \R_+$ we again may put $\underline{x}=r\underline{\omega}$, $y=\rho\underline{\nu}$ whereby $r^2=1+\rho^2,$ $\rho\in\R_+$ and $\underline{\omega}, \underline{\nu} \in S^{n-1}$ with $\underline{\omega}\perp\underline{\nu}$. This leads to the morphological bill
\begin{eqnarray*}
\C S^{n-1} & = & S^{n-1} + V_{n,2}(\R)\cdot\R_+ \\
           & = & S^{n-1} + S^{n-1}\cdot S^{n-2}\cdot\R_+ = S^{n-1}\cdot\left(1 + S^{n-2}\cdot\R_+\right)\\
           & = & S^{n-1}\cdot\left(1 + \left(\R^{n-1}-1\right)\right) = S^{n-1}\cdot\R^{n-1},
\end{eqnarray*}
which represents the tangent bundle to $S^{n-1}$. Once again remark that $S^{n-1}\cdot\R^{n-1}$ might represent any $(n-1)$-dimensional vector bundles over $S^{n-1}$ or more general stuff, so it only represents the quantity of the tangent bundle.\\
For $\overline{\C S}^{n-1}$ we have two approaches. First it is the set of points $(z_1,\ldots,z_{n+1})\in\C\BP^n$ solving the equation $z_1^2+\cdots+z_{n+1}^2=0$, which means that the homogeneous coordinates $(z_1,\ldots,z_{n+1})\neq 0$ belong to $NC^n\setminus\{0\}$ and they are determined up to a homogeneity factor $\lambda\in\C\setminus\{0\}.$ This leads to
\begin{gather*}
\overline{\C S}^{n-1} = \frac{NC^n-1}{\C-1} = \frac{V_{n+1,2}(\R)\cdot\R_+}{S^1\cdot\R_+}=\frac{V_{n+1,2}(\R)}{S^1}\\
= \widetilde{G}_{n+1,2}(\R)=\frac{S^n\cdot S^{n-1}}{S^1}.
\end{gather*}
Secondly we also have that
\begin{gather*}
\overline{\C S}^{n-1} = \C S^{n-1} + \overline{\C S}^{n-2}\\
S^{n-1}\cdot\R^{n-1} + S^{n-2}\cdot\R^{n-2} +\cdots + S^{1}\cdot\R^{} + 2
\end{gather*}
giving the total quantity, while also
\begin{gather*}
\C S^{n-1} + \overline{\C S}^{n-2} = S^{n-1}\cdot\left( \R^{n-1} + \frac{S^{n-2}}{S^1}\right)\\
S^{n-1}\cdot\frac{ \left((2\R+2)\R^{n-1}+S^{n-2} \right)}{S^1} = S^{n-1}\cdot\frac{S^n}{S^1},
\end{gather*}
as expected.\\
Note also that there is a 2-periodicity expressed by
$$
\overline{\C S}^{2n-1} = \frac{S^{2n}\cdot S^{2n-1}}{S^1} = S^{2n}\cdot\C\BP^{n-1},
$$
$$
\overline{\C S}^{2n} = \frac{S^{2n+1}\cdot S^{2n}}{S^1} = \C\BP^{n}\cdot S^{2n}.
$$
Note that we also have the identity
$$
NC^{n-1} = 1 + \overline{\C S}^{n-2}\cdot(\C-1)
$$
that often turns out useful in calculations.\\
We now use a purely complex method to compute the complexified sphere; we use a different  notation $\overline{\C \BS}^{n}$.\\
For $\overline{\C \BS}^{0}$ we have the equation
$$
z_1^2 + z_2^2 = 0 \;\;\Leftrightarrow\;\;uv=0,\;\; u=z_1+\mathbf{i}z_2, \; v=z_1-\mathbf{i}z_2.
$$
Up to a factor $\lambda\neq0$ there are solutions $(u,v)$ namely $(1,0)$ and $(0,1)$, leading to the quantity $\overline{\C \BS}^{0}=2.$\\
For $\overline{\C \BS}^{1}$ we have the equation
$$
uv=z^2_3
$$
including for $z^2_3=0,$ $uv=0$, i.e., $\overline{\C \BS}^{0}$ and for $z_3\neq0$ we normalise $z_3=1,$ so we have the equation for
$\overline{\C \BS}^{1}\;:\;uv=1,$ i.e., $u\in\C\setminus \{0\}$, $v=1/u$. This leads to
$$
\C\BS^1=\C-1\;\;\&\;\; \overline{\C \BS}^1 = \C\BS^1 + \C\BS^0 = (\C-1)+2 = \C+1
$$
so in fact $\overline{\C \BS}^1 = \C\BP^1 = \BS^2.$ \\
For $\overline{\C \BS}^2$ we again have the splitting
$$
\overline{\C \BS}^2 = \C \BS^2 + \overline{\C \BS}^1
$$
whereby $\overline{\C \BS}^2$ is given by the equation
$$
uv=z_4^2-z_3^2 = 1 - z_3^2,
$$
with normalization $z_4=1.$
There are two cases of this: $z_3^2\neq1,$ giving $z_3\in \C\setminus\{1,-1\}$ and $z_3\in\{+1,-1\}$. So, morphologically, we have a factor
$$
z_3\in\C-2\;\;\; \mathrm{or} \;\;z_3\in 2.
$$
In the case $z_3\in\C-2$ we have the equation
$$
uv=\mathrm{cte}\neq0
$$
to solve, which gives us $u\in\C-1$, $v=\mathrm{cte}/u$, leading to the quantity:
$$
(\C-1)(\C-2).
$$
For $z_3\in 2$ we have equation $uv=0$ to be solved, which gives us $(u,v)=(0,0)$ or $v=0$ and $u\in\C-1$ or $u=0$ and $u\in\C-1.$ So the total quantity is
$$
1 + 2 (\C-1),
$$
with an extra factor $2$, which gives the total
\begin{eqnarray*}
\C \BS^2 & = & (\C-1)(\C-2) + 4(\C-1) + 2\\
& = & (\C-1)(\C+2)  + 2 = \C^2 + \C = (\C+1)\C.
\end{eqnarray*}
Hence, we arrive at
$$
\overline{\C S}^{2} =  (\C+1)\C + (\C+1) = (\C+1)^2 = \BS^2\C\BP^1=\frac{S^3\BS^2}{S^1}.
$$
For $\C \BS^3$ we have the equation
$$
u_1v_1 = 1-u_2v_2
$$
leading to the cases $u_2v_2 = 1$ and $u_2v_2 \neq 1$ for which we have the morphological factors $\C-1$ and $\C^2-\C+1$ (the phantom complex projective plane). In case $1-u_2v_2 = c\neq 0$ the remaining equation $u_1v_1 = c$ yields the factor $\C-1$ while for $1 = u_2v_2$ we have $u_1v_1 = 0$, i.e., $1+2(\C-1).$ In total this gives
\begin{eqnarray*}
\C \BS^3 & = & (\C-1)(\C^2-\C+1) + (1+2(\C-1))(\C-1)\\
& = & (\C-1)(\C^2+\C) = (\C^2 -1)\C,
\end{eqnarray*}
which is also clear from the fact that $u_1v_1 + u_2v_2 = 1$ is basically the equation $ad-bc = 1$ for
$$
SL(2,\C)=\left(\C^2-1\right)\C.
$$
This leads to
\begin{eqnarray*}
\overline{\C \BS}^3 & = & \C \BS^3  + \overline{\C \BS}^2 = \left(\C^2-1\right)\C + \left(\C+1\right)^2\\
& = & (\C+1)\left( (\C-1)\C+\C+1 \right) = (\C+1)(\C^2+1)\\
& = & \frac{\BS^4 S^3}{S^1} = \BS^4\cdot \C \BP^1 = \C \BP^3.
\end{eqnarray*}
For $\C \BS^4$ we have the equation
$$
u_1v_1 + u_2v_2 = 1 - z_5^2
$$
leading to the factors $\C-2$ for $z_5^2\neq1$ and $2$ for $z_5^2=1$.

For $1-z_5^2=c\neq0$ the remaining equation gives the factor $SL(2,\C) = (\C^2-1)\C$ while for $c=0$ we have the equation $u_1v_1 + u_2v_2 = 0$, which is the nullcone
$$
N\C^3 = 1 + \overline{\C \BS}^2\cdot(\C-1) = 1 + (\C^2-1)(\C+1).
$$
In total we get
\begin{eqnarray*}
\C \BS^4 & = & (\C^2-1)\C(\C-2) + 2(\C^2-1)(\C+1)+2\\
& = & (\C^2-1)(\C^2+2) + 2 = \C^4+\C^2 = \C^2(\C^2 +1)
\end{eqnarray*}
so that
\begin{eqnarray*}
\overline{\C \BS}^4 & = & \C \BS^4  + \overline{\C \BS}^3 = \left(\C^2+1\right)\left(\C^2+\C+1\right)\\
& = & \C\BP^2\cdot\BS^4 = \frac{S^5\BS^4 }{S^1}.
\end{eqnarray*}
It seems that in general we will have
\begin{eqnarray*}
\overline{\C \BS}^{2n} & = & \frac{S^{2n+1}}{S^1} \BS^{2n} = \C\BP^n\cdot \BS^{2n},
\end{eqnarray*}
\begin{eqnarray*}
\overline{\C \BS}^{2n-1} & = & \BS^{2n}\frac{S^{2n-1}}{S^1}  = \BS^{2n}\cdot\C\BP^{n-1} = \C\BP^{2n-1}.
\end{eqnarray*}
To prove this recursively we begin with $\C \BS^{2n}$ given by the equation
$$
u_1v_1 + \cdots + u_nv_n = 1 - z_{2n+1}^2.
$$
For the right hand side we have the factor $\C-2$ for $z_{2n+1}^2\neq1$ and the factor $2$ for $z_{2n+1}^2=1$. The equation $c = 1 - z_{2n+1}^2\neq0$ gives the factor $u_1v_1 + \cdots + u_nv_n = c,$ which is in fact $\C\BS^{2n-1}$ while for $c=0$ we have the equation $u_1v_1 + \cdots + u_n v_n = 0,$ which is
$$
N\C^{2n-1} = 1 + \overline{\C \BS}^{2n-2}\cdot(\C - 1).
$$
So in total we have
\begin{eqnarray*}
\C \BS^{2n} & = & \C \BS^{2n-1} \cdot(\C-2) + 2+ 2\overline{\C \BS}^{2n-2}\cdot(\C-1)\\
           & = & \overline{\C \BS}^{2n-1}\cdot(\C-2) - \overline{\C \BS}^{2n-2}\cdot(\C-2) + 2\overline{\C \BS}^{2n-2}\cdot(\C-1) + 2\\
           & = & \overline{\C \BS}^{2n-1}\cdot(\C-2) + \overline{\C \BS}^{2n-2}\cdot\C + 2
\end{eqnarray*}
and, therefore,
$$
\overline{\C \BS}^{2n} = \C \BS^{2n} + \overline{\C\BS}^{2n-1} = \overline{\C \BS}^{2n-1}\cdot(\C-1) + \overline{\C \BS}^{2n-2}\cdot\C + 2.
$$
Using the induction hypothesis $\overline{\C \BS}^{2n-1} = \C\BP^{2n-1}$ and $\overline{\C S}^{2n-2} = \C\BP^{n-1}\cdot\BS^{2n-2}$, this gives rise to
\begin{eqnarray*}
\overline{\C \BS}^{2n} & = & \C^{2n}-1 + (\C^{n-1}+1)(\C^{n}+\C^{n-1} + \cdots + \C) + 2\\
           & = & (\C^{2n} + \C^{2n-1} + \cdots + \C^n) + (\C^{n} + \C^{n-1}+ \cdots +\C+1)\\
           & = & (\C^{n}+1)\C\BP^n = \C\BP^n \cdot \BS^{2n}.
\end{eqnarray*}
For the other case $\overline{\C \BS}^{2n+1}$ we note that $\C\BS^{2n+1}$ is given by the equation $u_1v_1 + \cdots + u_nv_n = 1 - u_{n+1}v_{n+1}$ giving the factor (Phantom projective plane) $\C^2 - \C + 1$ for $c = 1 - u_{n+1}v_{n+1} \neq 0$ and $\C-1$ for $u_{n+1}v_{n+1}=1$. Again for $c\neq0$ we have the equation $u_1v_1 + \cdots + u_nv_n = c\neq0$, leading to the factor $\C\BS^{2n-1}$ and for $c=0$ we get factor $N\C^{2n-1}$ as before. This leads to
\begin{eqnarray*}
\C\BS^{2n+1} & = & \C\BS^{2n-1}\cdot(\C^2-\C+1) + (1 + \overline{\C \BS}^{2n-2}\cdot(\C-1))\cdot(\C-1)\\
            & = & \overline{\C \BS}^{2n-1}\cdot(\C^2-\C+1) - \overline{\C \BS}^{2n-2}\cdot\C + \C -1,
\end{eqnarray*}
which, using the formulae for $\overline{\C \BS}^{2n-1}$ and $\overline{\C S}^{2n-2}$ yields
\begin{eqnarray*}
\C\BS^{2n+1} & = & \C^n\cdot(\C^{n+1}-1)\\
            & = & S^{2n+1}\cdot\R^{2n}\cdot\R_+
\end{eqnarray*}
and so we finally get
\begin{eqnarray*}
\overline{\C \BS}^{2n+1} & = & \C\BS^{2n+1} + \overline{\C \BS}^{2n}\\
            & = & \C^{n}\cdot(\C^{n+1}-1) + \left(\C^{n}+1\right)\left(\C^n + \cdots + 1\right) \\
            & = & \C^{2n+1} + \C^{2n} + \cdots + \C^{n} - \C^{n} + \C^{n}+ \cdots +1 = \C\BP^{2n+1} \\
           & = & \BS^{2n+2}\cdot\C\BP^n = \BS^{2n+2}\cdot\frac{S^{2n+1}}{S^1}.
\end{eqnarray*}
These calculations show a certain consistency in which the Poincar\'e sphere $\BS^{2n}$ and complex projective spaces $\C\BP^n$ play a central role. Also the Phantom complex projective plane $\C^2-\C+1$ reappears here as the set of points $(u,v)\in\C^2$ which lie outside the hyperbola $uv=1;$ it is the new geometric interpretation for strange phantom plane that arises from the morphological analysis.\\
Finally also the ``bipolar plane'' $\C-2$ arises naturally within the discussion. Of course one could always consider $\C-n$, but in morphological calculus we are not interested in generality, only in canonical objects.\\

Our next investigation concerns ``Null Gra\ss mannians''.\\
By $NG_{n,k}(\C)$ we denote the manifold of all $k$-dimensional subspaces of the nullcone $NC^{n-1}$ in $\C^n$. Hence in particular $NG_{n,1}(\C) = \overline{\C S}^{n-2}$. Let us make the morphological analysis; once again there are two ways.

Let $V\subset NC^{n-1}$ be a $k$-dimensional complex subspace spanned by $k$-vectors $\underline{\tau}_1, \ldots,\underline{\tau}_k$. Which are of course linearly independent and satisfy:
\[
\underline{\tau}_j^2 = (\underline{t}_j + \mathbf{i}\underline{s}_j)^2=0,\;\;\mathrm{i.e.,}\;\;\underline{t}_j\perp\underline{s}_j\;\&\;|\underline{t}_j|=|\underline{s}_j|
\]

\[
\langle\underline{\tau}_j,\underline{\tau}_k\rangle = \langle\underline{t}_j,\underline{t}_k\rangle-\langle\underline{s}_j,\underline{s}_k\rangle+\mathbf{i}(\langle\underline{t}_j,\underline{s}_k\rangle + \langle\underline{t}_k,\underline{s}_j\rangle)=0.
\]

Next consider on $\C^n$ the Hermitian inner product $(\underline{z},\underline{w})=\sum_{j=1}^{n}\underline{z}_j\underline{\overline{w}}_j$; then we can normalize vector $\underline{\tau}_1$, i.e., $(\underline{\tau}_1,\underline{\overline{\tau}}_1) = \langle\underline{t}_1,\underline{t}_1\rangle + \langle\underline{s}_1,\underline{s}_1\rangle = 2$, which together with $\underline{t}_1\perp\underline{s}_1\;|\underline{t}_1|=|\underline{s}_1|$ means that the pair $(\underline{t}_1,\underline{s}_1)\in V_{n,2}(\R)$ is the manifold of orthonormal $2$-frames.

Next one may choose $\underline{\tau}_2$ such that $(\underline{\tau}_2,\underline{\tau}_1) = 0$ \& $|\underline{\tau}_2|^2=2$ with $\underline{\tau}_2=\underline{t}_2 + \mathbf{i}\underline{s}_2$. This automatically implies that
$$
\langle\underline{\overline{\tau}}_2,\underline{\tau}_1\rangle = \langle\underline{\tau}_2,\underline{\tau}_1\rangle = 0,\; \mathrm{i.e.,} \;\langle\underline{t}_2,\underline{\tau}_1\rangle = \langle\underline{s}_2,\underline{\tau}_1\rangle =0
$$
so that the pair $(\underline{t}_2,\underline{s}_2)$ is an orthonormal $2$-frame that is also orthogonal to span$_\R\{\underline{t}_1,\underline{s}_1\}$, i.e., $(\underline{t}_1,\underline{s}_1,\underline{t}_2,\underline{s}_2)\in V_{n,4}(\R).$

Continuing the reasoning, we may choose $\underline{t}_j, \underline{s}_j$ in such a way that

$$(\underline{t}_1,\underline{s}_1,\underline{t}_2,\underline{s}_2,\ldots,\underline{t}_k,\underline{s}_k)\in V_{n,2k}(\R) = \frac{SO(n)}{SO(n-2k)};$$
a necessary condition for this is $n\geq 2k$.\\

Now let $(\underline{\tau}'_1,\ldots, \underline{\tau}'_k)$ be another $k$-tuple for which
$$
\mathrm{span}\{\underline{\tau}'_1,\ldots, \underline{\tau}'_k\} = V\; \&\; |\underline{\tau}'_j|^2=2,\; (\underline{\tau}'_j,\underline{\tau}'_k) = 0,\; j\neq k;
$$
then there exists the unique matrix $A\in U(k)$ such that $\underline{\tau}'_j = \sum_{\ell=1}^{k}A_{j\ell}\underline{\tau}_\ell$. Hence we obtain the identity in terms of homogeneous spaces and in morphological sense

$$
NG_{n,k}(\C) = \frac{SO(n)}{U(k)\times SO(n-2k)} = \frac{S^{n-1}\cdot S^{n-2}\cdots S^{n-2k}}{S^{2k-1}\cdot S^{2k-3}\cdots S^{1}}.
$$
So, in the case $n=2m$ is even, we have that
$$
NG_{n,k}(\C) = \frac{S^{2m-1}\cdots S^{2m-2k}}{S^{2k-1}\cdots S^{1}} = G_{m,k}(\C)\cdot S^{2m-2}\cdots S^{2m-2k}
$$
while for $n=2m+1$, odd, we have
$$
NG_{n,k}(\C) = \frac{S^{2m}\cdot S^{2m-1}\cdots S^{2m-2k+1}}{S^{2k-1}\cdots S^{1}} = G_{m,k}(\C)\cdot S^{2m}\cdots S^{2m-2k+2}.
$$
This also implies that $NG_{m,k}(\C)$ is integrable.\\

Another way of calculating the quantity makes use of the complex compact spheres $\overline{\C\BS}^{n-2}$ that were obtained in terms of complex analysis. Using the notation $N\G_{n,k}(\C)$ for the corresponding null Gra\ss mannians we have:
\begin{eqnarray*}
  N\G_{n,1}(\C) & = & \overline{\C\BS}^{n-2},\\
  N\G_{n,2}(\C) & = & \{(\underline{\tau}_1,\underline{\tau}_2)\in \overline{\C\BS}^{n-2}\cdot S^1\times \overline{\C\BS}^{n-4}\cdot S^1\}\,\mathrm{mod}\,U(2)\\
                & = & \frac{\overline{\C\BS}^{n-2}\cdot S^1\cdot \overline{\C\BS}^{n-4}\cdot S^1}{S^3\cdot S^1} = \frac{\overline{\C S}^{n-2}\cdot \overline{\C S}^{n-4}}{\C\BP^1},
\end{eqnarray*}
and in general
\begin{eqnarray*}
  N\G_{n,k}(\C) & = & \frac{\overline{\C S}^{n-2}\cdots \overline{\C S}^{n-2k}}{\C\BP^{k-1}\cdots\C\BP^1}.
\end{eqnarray*}
Hence, in case $n=2m$ we obtain
\begin{eqnarray*}
  N\G_{2m,k}(\C) & = & \frac{\overline{\C S}^{2m-2}\cdots \overline{\C S}^{2m-2k}}{\C\BP^{k-1}\cdots\C\BP^1}\\
                 & = & \frac{\C\BP^{m-1}\cdot\BS^{2m-2}\cdots\C\BP^{m-k}\cdot\BS^{2m-2k}}{\C\BP^{k-1}\cdots\C\BP^1}\\
                 & = & \frac{\C\BP^{m-1}\cdots\C\BP^{m-k}}{\C\BP^{k-1}\cdots\C\BP^1}\BS^{2m-2}\cdots\BS^{2m-2k}\\
                 & = &  G_{m,k} (\C ) \cdot\BS^{2m-2}\cdots\BS^{2m-2k}
\end{eqnarray*}
and similarly for $n=2m+1$ we get
$$
N\G_{2m+1,k}(\C) =  G_{m,k} (\C ) \cdot\BS^{2m}\cdots\BS^{2m-2k+2},
$$
and so these objects are also integrable.
The calculus of nullcones and things can also be done in  real variables. Let $\R^{p,q}$ be the space $\R^{p,q} = \R^{p+q}$ with quadratic form
$$
|\underline{x}|^2 - |\underline{y}|^2 =  \sum_{j=1}^{p}x_j^2 - \sum_{j=1}^{q}y_j^2, \;\; (\underline{x},\underline{y})\in\R^{p,q}.
$$
Then the nullcone $NC^{p,q}$ is the set of solutions $(\underline{x},\underline{y})$ of equation $|\underline{x}|^2 = |\underline{y}|^2$; it contains  of course $(0,0)$ and for $|\underline{x}|\in\R_+$ we have $(\underline{x},\underline{y})=\rho (\underline{\omega},\underline{\nu})$ with $\rho>0$ and $(\underline{x},\underline{y})\in S^{p-1}\times S^{q-1}$. Hence, we have relation
$$
NC^{p,q} = S^{p-1}\cdot S^{q-1}\cdot\R_+ + 1.
$$

By $S^{p-1,q-1}$ we denote the set of $1D$ subspaces of $NC^{p,q}$; it may be represented by the equivalence classes $(\underline{\omega},\underline{\nu})\sim (-\underline{\omega},-\underline{\nu})$, $(\underline{\omega},\underline{\nu})\in S^{p-1}\times S^{q-1}.$ In morphological notation we have:
$$
S^{p-1,q-1} = \frac{NC^{p,q}-1}{\R-1} = \frac{S^{p-1}\cdot S^{q-1}\cdot\R_+}{\R-1} = \frac{S^{p-1}\cdot S^{q-1}}{2} = S^{p-1}\cdot \R\BP^{q-1}.
$$
For example for $q=2$ we may put $\underline{\nu}=(\cos\theta,\sin\theta)$ and $S^{p-1,1}$ may be identified with the equivalent pairs $(\underline{\omega},\cos\theta,\sin\theta)\sim (-\underline{\omega},-\cos\theta,-\sin\theta)$, which is equivalent with the Lie sphere
$$
S^{p-1,1} \cong LS^p = \{e^{\mathbf{i}\theta}\underline{\omega}\,:\,\underline{\omega}\in S^{p-1},\;\theta\in [0,\pi[\}.
$$
But there is also another calculation of this manifold that leads to another quantity $\BS^{p-1,q-1}$ and it corresponds to the ``conformal compactification'' $\overline{\R}^{p-1,q-1}$ of $\R^{p-1,q-1}$. To find this, let $(\underline{x},\underline{y}) = (\underline{x}',x_p;\underline{y}',y_q)$ with $(\underline{x}',\underline{y}')\in \R^{p-1,q-1}$. Then first we may intersect the nullcone with the plane $x_p - y_q = 1$, i.e., we put
$$
x_p = \frac{1}{2}(1-\rho),\;\; y_q = -\frac{1}{2}(1+\rho).
$$
The equation $|\underline{x}|^2 = |\underline{y}|^2$ for the manifold $\BS^{p-1,q-1}$ gives us $\rho = |\underline{x}'|^2 - |\underline{y}'|^2$ so that $(\underline{x}',\underline{y}')\in \R^{p-1,q-1}$ freely and then $(x_p,y_q)$ are fixed. So this part of $\BS^{p-1,q-1}$ is equivalent to $\R^{p+q-2}$. The remaining part of $\BS^{p-1,q-1}$ is represented by the nonzero vectors $\lambda(\underline{x},\underline{y})$, $\lambda\in\R\setminus\{0\}$ for which $x_p=y_q$; there are two cases:
\begin{itemize}
\item If $x_p=y_q\neq0$ we may normalize $x_p = y_q = 1$ and we have $(\underline{x},\underline{y})= (\underline{x}',1,\underline{y}',1)$ together with the equation $|\underline{x}'|^2 - |\underline{y}'|^2 = 0$. So this part of $\BS^{p-1,q-1}$ is equivalent to the modified nullcone
$$
N\C^{p-1,q-1} = 2\BS^{p-2,q-2}\cdot\R_++1.
$$
\item If $x_p=y_q=0$ we have $\lambda(\underline{x}',0,\underline{y}',0)$ with $\lambda\in\R\setminus\{0\}$ and $|\underline{x}'|\neq0$ and $|\underline{x}'|=|\underline{y}'|$, which is the definition of $\BS^{p-2,q-2}$.
\end{itemize}
So the total morphological calculation becomes
\begin{eqnarray*}
\BS^{p-1,q-1} & = & \R^{p+q-2} + \BS^{p-2,q-2} (2\R_++1)+1,\\
             & = & \R^{p+q-2} + \BS^{p-2,q-2}\cdot\R+1,
\end{eqnarray*}
or, in terms of compactification of $\R^{p,q}$:
$$
\overline{\R}^{p,q} = \R^{p+q} + \overline{\R}^{p-1,q-1}\cdot\R+1.
$$
\begin{description}
\item [case 1]: for $q=0$ we simply obtain
$$
\overline{\R}^{p,0} = \overline{\R}^{p} = \R^p + 1 = \BS^p.
$$
\item [case 2]: compactified Minkowski space-time
\begin{eqnarray*}
\overline{\R}^{p,1} & = & \R^{p+1} + \overline{\R}^{p-1,0}\cdot\R+1,\\
                    & = & \R^{p+1} + (\R^{p-1}+1)\cdot\R+1 = (\R^{p}+1)(\R+1)\\
                    & = & \BS^{p}\cdot \R\BP^1.
\end{eqnarray*}
\end{description}
More in general we obtain for $p\geq q$
\begin{eqnarray*}
\overline{\R}^{p,2} & = & \R^{p+2} + \overline{\R}^{p-1,1}\cdot\R+1,\\
                    & = & \R^{p+2} + (\R^p+\R^{p-1}+\R+1)\cdot\R+1 = (\R^{p}+1)(\R+1)\\
                    & = & \R^{p+2} + \R^{p+1} + \R^{p} + \R^{2}+ \R + 1\\
                    & = & (\R^{p} +1)(\R^{2}+ \R + 1) = \BS^p\cdot \R\BP^2,
\end{eqnarray*}
and, continuing in this way we obtain for $p\geq q:$
$$
\overline{\R}^{p,q} = (\R^{p} +1)(\R^q+\cdots+\R^{2}+ \R + 1) = \BS^p\cdot \R\BP^q,
$$
as expected from the similar (but different) formula $S^{p,q} = S^p\cdot\R\BP^q.$

So once again we have two different quantities that are obtained in two different canonical ways from what is mathematically considered to be one manifold.

Note that in particular (and this is weird)
\begin{eqnarray*}
\overline{\R}^{m,m} & = & (\R^{m} +1)(\R^{m}+ \cdots+\R + 1) \\
                    & = & \R^{2m}+\cdots+\R^{m+1} + 2\R^{m} + \R^{m-1} + \cdots + 1 \\
                    & = & (\R^{m} +1)^2+\R\cdot(\R^{m} + 1)(\R^{m-2}+\cdots+ 1).
\end{eqnarray*}
For the classical Minkowski space time we get
$$
\overline{\R}^{3,1} = (\R^3+1)(\R+1),
$$
and this is indeed projective line bundle over $3$-sphere.

Compactified complexified Minkowski space-time is given by
\begin{eqnarray*}
\overline{\C\BS}^4 &=& (\C^2+1)(\C^2 + \C + 1) = \C^4 +\C^3 +2\C^2 +\C + 1\\
                  &=& \C^4 + \C\cdot(\C+ 1)^2+1 = \C^4 + N\C^3,
\end{eqnarray*}
with $N\C^3 = \C\cdot \overline{\C\BS}^2 + 1,$
so it is not just replacing ``$\R$'' by ``$\C$'' in $\overline{\R}^{3,1}$.

We must still calculate the null Gra\ss mannians $NG_{p,q;k}$; they are defined as manifold of $k$-dimensional subspaces of the nullcone $N\C^{p,q} = S^{p-1}\cdot S^{q-1}\cdot\R^+ + 1.$

Let $V$ be such $k$-dimensional plane; then $V$ is spanned by the basis of the form:
$$
e_1+\epsilon_1, e_2+\epsilon_2,\ldots,e_k+\epsilon_k;\; e_1,\ldots ,e_k\in S^{p-1};\; \epsilon_1,\ldots ,\epsilon_k\in S^{q-1};
$$
orthonormal frames, so these bases belong to:
$$
e_1+\epsilon_1\in S^{p-1}\cdot S^{q-1} = 2\frac{N\C^{p,q}-1}{\R-1} = 2S^{p-1,q-1},
$$
up to
$$
e_k+\epsilon_k\in S^{p-k}\cdot S^{q-k} = 2S^{p-k,q-k},
$$
and within $V$ the total quantity of such bases is given by $O(k) = S^{k-1}\cdot S^{k-2}\cdots S^0.$ We thus have the morphological representation (with $p\geq q, \; q\geq k$)
\begin{eqnarray*}
NG_{p,q;k} &=& \frac{\left(2S^{p-1,q-1}\right)\cdots\left(2S^{p-k,q-k}\right)}{S^{k-1}\cdot S^{k-2}\cdots S^0}\\
                  &=& \frac{S^{p-1,q-1}\cdots S^{p-k,q-k}}{\R\BP^{k-1}\cdots\R\BP^{1}}\\
                  &=& \frac{S^{p-1}\cdot S^{p-2} \cdots S^{p-k}}{S^{k-1}\cdots S^{0}}S^{q-1}\cdots S^{q-k}\\
                  &=& G_{p,k}(\R)\cdot S^{q-1}\cdots S^{q-k}.
\end{eqnarray*}

Again there is another way of computing this whereby in the above, $S^{p,q}$ is replaced by $\BS^{p,q} = \overline{\R}^{p,q}$, leading up to the stereographic null-Gra\ss mannian:
\begin{eqnarray*}
N\G_{p,q;k} & = & \frac{\left(2\BS^{p-1,q-1}\right)\cdots\left(2\BS^{p-k,q-k}\right)}{S^{k-1}\cdots S^{0}}\\
            & = & \frac{\overline{\R}^{p-1,q-1}\cdots \overline{\R}^{p-k,q-k}}{\R\BP^{k-1}\cdots \R\BP^1}\\
            & = & \frac{\BS^{p-1}\cdots \BS^{p-k}\cdot \R\BP^{q-1}\cdots \R\BP^{q-k}}{\R\BP^{k-1}\cdots \R\BP^1} = G_{q,k}(\R)\cdot \BS^{p-1}\cdots \BS^{p-k}.
\end{eqnarray*}

All these manifolds give hence rise to integrable morphological objects.

For the Minkowski space-time we have the manifold of null-lines (light rays):

\begin{eqnarray*}
N\G_{4,2;2} & = & \frac{\overline{\R}^{3,1}\cdot \overline{\R}^{2,0}}{\R\BP^1} = \frac{(\R^3+1)(\R+1)(\R^2+1)}{\R+1}\\
            & = & (\R^3+1)(\R^2+1) = \BS^3\cdot \BS^2,
\end{eqnarray*}
i.e., the real twistor space.

The last case we consider here is that of the space $\C^{p,q} = \C^{p+q}$ provided with the pseudo-Hermitian form
$$
((\underline{z},\underline{u}),(\underline{z}',\underline{u}')) = (\underline{z},\underline{z}') - (\underline{u},\underline{u}') = \sum_{j=1}^{p}z_j \overline{z}'_j - \sum_{j=1}^{q}u_j \overline{u}'_j.
$$
The nullcone $((\underline{z},\underline{u}),(\underline{z},\underline{u})) = 0$ is denoted by $NC^{p,q}(\C)$ and it has real codimension one, so its real dimension equals $2p+2q-1.$ The equation is $|\underline{z}|^2 = |\underline{u}|^2$ so:
$$
NC^{p,q}(\C) = S^{2p-1}\cdot S^{2q-1}\cdot \R_+ + 1.
$$
By $T^{p,q}$ we denote the manifold of one dimensional complex subspaces of $NC^{p,q}(\C);$ it is a real submanifold of $\C\BP^{p+q-1}$ of real codimension one that hence subdivides $\C\BP^{p+q-1}$ in $3$ parts and  $T^{2,2}\subset\C\BP^3$ corresponds to ``real twistor space'' (the manifold of light-lines in Minkowski space). From the definition we have
$$
T^{p,q} = \frac{NC^{p,q}(\C)-1}{\C-1} = \frac{S^{2p-1}\cdot S^{2q-1}\cdot\R_+}{S^1\cdot\R_+} = S^{2p-1}\cdot \C\BP ^{q-1},
$$
so in particular $T^{2,2} = S^3\cdot \C\BP^1 = S^3\cdot \BS^2.$

There is another approach leading to the twister space $\T^{p,q}$ with a different quantity.

To that end we write $(\underline{z},\underline{u}) = (\underline{z}',z_p,\underline{u}', u_p)$ and consider the intersection $NC^{p,q}(\C)\cap\{z_p-u_q = 1\}$, which allows us to write
$$
z_p = \frac{1}{2}(1+\rho+\mathbf{i}\alpha),\;\; u_q = \frac{1}{2}(-1+\rho+\mathbf{i}\alpha).
$$
The equation for the point $(\underline{z},\underline{u})$ now becomes
$$
|\underline{z}'|^2-|\underline{u}'|^2 + \frac{1}{4}\left((1+\rho)^2+\alpha^2\right) - \frac{1}{4}\left((1-\rho)^2+\alpha^2\right) = 0
$$
or
$$
\rho=|\underline{z}'|^2-|\underline{u}'|^2 \; \&\;\alpha\in\R.
$$
Hence the $1D$ complex subspaces of $NC^{p,q}(\C)$ that intersect the plane $z_p - u_q = 1$ are representable by vectors of the form $\left(\underline{z}',\frac{1}{2}(1+\rho+\mathbf{i}\alpha), \underline{u}',\frac{1}{2}(-1+\rho+\mathbf{i}\alpha)\right)$ with $(\underline{z}',\underline{u}')\in \C^{p-1,q-1}$ and $\alpha\in\R.$ So this part of $\T^{p,q}$ has quantity $\C^{p+q-2}\cdot \R.$ The other points of $\T^{p,q}$ have the form $(\underline{z}',\lambda,\underline{u}', \lambda)$ so there are two cases
\begin{itemize}
\item $\lambda\neq0,$ in this case we normalize $\lambda=1$ and we have the equation $|\underline{z}'|^2-|\underline{u}'|^2=0,$ giving a version of nullcone:
$$
N\C^{p-1,q-1}(\C) = \T^{p-1,q-1}\cdot (\C-1) + 1
$$
\item In the case $\lambda=0$ we have the point $(\underline{z}',0,\underline{u}', 0)$ with equation $|\underline{z}'|=|\underline{u}'|$ and determined up to a constant $c\in\C\setminus\{0\}$, i.e., we get $\T^{p-1,q-1}$.
\end{itemize}

This leads to the recursion formula for $\T^{p,q}$ with $p\geq q:$
$$
\T^{p,q} = \C^{p+q-2}\cdot\R + \T^{p-1,q-1}\cdot \C + 1.
$$
So in particular we get

\begin{eqnarray*}
\T^{p,1} & = & \C^{p-1}\cdot\R + 1 = \R^{2p-1} + 1= \BS^{2p-1},\\
\T^{p,2} & = & \C^{p}\cdot\R + \left(\C^{p-2}\cdot\R + 1\right)\C + 1 \\
         & = & \left(\C^{p} + \C^{p-1}\right)\R + \C + 1 = \left(\C^{p-1}\cdot\R + 1\right)(\C + 1)\\
         & = & \BS^{2p-1}\cdot \C\BP^1,
\end{eqnarray*}

\begin{eqnarray*}
\T^{p,3} & = & \C^{p+1}\cdot\R + \left(\C^{p-1}\cdot\R + \C^{p-2}\cdot\R+\C+ 1\right)\C + 1 \\
         & = & \left(\C^{p-1}\cdot\R + 1\right)\left(\C^{2} + \C +1\right) = \BS^{2p-1}\cdot \C\BP^2
\end{eqnarray*}
and so, continuing in this way, we obtain for $p\geq q$:

\begin{eqnarray*}
\T^{p,q} & = & \left(\C^{p-1}\cdot\R + 1\right)\left(\C^{q-1} + \cdots+ \C +1\right) = \BS^{2p-1}\cdot \C\BP^{q-1}.
\end{eqnarray*}
In particular we re-obtain the expected formula
$$
\T^{2,2} = (\C\cdot\R+1)(\C+1) = \BS^{3}\cdot \C\BP^{1} = \BS^3\cdot\BS^2.
$$
The manifold $\C\BP^{3}$ is itself called the complex twistor space; it decomposes into real twistor space $\T^{2,2}$ together with two equals parts corresponding to $|\underline{z}|<|\underline{u}|$ and $|\underline{z}|>|\underline{u}|$. In morphological language we have the cutting experiment
\begin{eqnarray*}
\C\BP^3 - \T^{2,2} & = & \C^3 + \C^2 + \C + 1 - (\C\R+1)(\C+1) \\
                  & = & (\C^2-\C\cdot\R)(\C+1) = \C\cdot (\C-\R)(\C+1)\\
                  & = & 2\C\cdot\C_+\cdot (\C+1),
\end{eqnarray*}
which indeed gives $2$ copies of $\C\cdot\C_+\cdot (\C+1)$ whereby we put $\C_+ =\frac{\C-\R}{2}= \R\cdot\frac{\R-1}{2} = \R\cdot\R_+.$

We can also calculate the null-Gra\ss mannian $NG^{p,q;k}(\C)$ of $k$-dimensional complex subspaces of $NC^{p,q}(\C)$. Let $V$ be a complex $k$-subspace; then the frames $(\underline{t}_1;\underline{s}_1),\ldots, (\underline{t}_k;\underline{s}_k)$ may be chosen such that $(\underline{t}_j;\underline{t}_k) - (\underline{s}_j;\underline{s}_k) = 0$, of course, but we may also choose $(\underline{t}_j;\underline{t}_j)$ to be the Hermitian orthonormal frame, i.e., $(\underline{t}_j;\underline{t}_k) + (\underline{s}_j;\underline{s}_k) = 0$ and $|\underline{t}_j|=|\underline{s}_j| = 1$.

So in fact we can choose
\begin{eqnarray*}
(\underline{t}_1;\underline{s}_1) &\in & S^{2p-1}\times S^{2q-1},\\
(\underline{t}_2;\underline{s}_2) &\in & S^{2p-3}\times S^{2q-3},
\end{eqnarray*}
and so on. Moreover these frames per plane $V$ can be chosen in $U(k)$-different ways, leading up to the morphological formula for $p\geq q \geq k$

\begin{eqnarray*}
NG_{p,q;k} (\C) & = & \frac{S^{2p-1}\cdots S^{2p-2k}\cdot S^{2q-1}\cdots S^{2q-2k}}{S^{2k-1}\cdots S^{3}\cdot S^{1}}\\
                & = & S^{2p-1}\cdots S^{2p-k}\cdot \frac{\C\BP^{q-1}\cdots \C\BP^{q-k}}{\C\BP^{k-1}\cdots \C\BP^{1}}.
\end{eqnarray*}
A similar calculation can be made using the stereographic spheres, leading to:
\begin{eqnarray*}
N\G_{p,q;k} (\C) & = & \frac{\left(\T^{p,q}\cdot\frac{(\C-1)}{\R_+}\right)\cdot\left(\T^{p-1,q-1}\cdot\frac{(\C-1)}{\R_+}\right)\cdots \left(\T^{p-k+1,q-k+1}\cdot\frac{(\C-1)}{\R_+}\right)}{U(k)}\\
                & = & \BS^{2p-1}\cdots \BS^{2p-2k}\cdot \frac{\C\BP^{q-1}\cdots \C\BP^{q-k}}{\C\BP^{k-1}\cdots \C\BP^{1}},
\end{eqnarray*}
and in particular for $p=q=2$, $k=2$ we obtain:

$$
N\G_{2,2;2} (\C) = \frac{\BS^3\cdot\BS^1\cdot \C\BP^1}{\C\BP^1} = (\R^3+1)(\R+1),
$$
which corresponds to the real compactified  Minkowski space.

The compactified complex Minkowski space corresponds to:
$$
G_{4,2} = \frac{S^7\cdot S^5}{S^3\cdot S^1} = \BS^4\cdot \C\BP^2 = (\C^2+1)(\C^2+\C+1) = \overline{\C\BS}^4,
$$
as can be shown using bivectors and Klein quadric.

\section{Conclusions and Remarks}

\begin{description}
\item[(i)] COMPLETENESS

Morphological calculus is best compared with a museum. It consists of a lots of special names, algebraic expressions and calculations that stand for geometrical objects and operations on these objects.

In this paper we presented morphological calculus for the most important classical manifolds. Like any museum, also our collection is incomplete. For example a full morphological treatment  for the spin groups $\mathrm{Spin}(m)$ and $\mathrm{Spin}(p,q)$ is still to be done and there is a vast collection of special manifolds or objects to be added to the catalogue.

In building up our museum we give preference to the most interesting special manifolds (canonical manifolds) as well as to the ``simplest ways of introducing them''. So in fact the calculus is entirely based on examples of objects and experiments; there is no idea of ``a general manifold'' and no theory behind the scene.
\\

\item[(ii)] CORRECTNESS

Morphological calculus is correct in the sense that it takes space within the language of calculus that is a correct language based on clear rules. This leads to the notion of quantity, which is in fact what a manifold becomes once it is introduced within the calculus language. This is practically done by assigning a name to an object along with an algebraic relation that expresses the definition of the object in calculus. The notion of quantity is somewhat comparable to the notions of cardinality and of volume that are used to express the contents or size of an object. But there is no mathematical definition for it; it is an imaginary substance that resides entirely within the calculus.

The main problem is not the calculus itself but the way of translating objects of geometry into calculus expressions (morphological analysis); it usually happens that one and the same object can be translated into morphological language in many ways and that may cause confusion.

To give an example, the compactified Minkowski space is given by
$$
\overline{\R}^{3,4} = \R^4 + \R(\R^2+1) + 1 = (\R^3 + 1)(\R + 1)
$$
whereby $\R^4$ is the usual Minkowski space and
$$
\R(\R^2+1)+1 = (2\R_+\BS^2 + 1) + \BS^2
$$
is a compactified light cone at infinity whereby we made use of stereographic sphere $\BS^2$. What would happen if we replace $\BS^2$ by the usual sphere $S^2$ Well, we would get in total:

\begin{eqnarray*}
\R^4 + \R(2\R^2 + 2\R + 2) + 1 &=& \R^4 + 2\R^3 + 2\R^2 + 2\R + 1 \\
 =(\R^3 + \R^2 + \R + 1)(\R+1)  &=& \R\BP^3\cdot\frac{S^1}{2} = \frac{S^3}{2}\frac{S^1}{2} \\
 =\BS^2\cdot \BS^1\cdot \BS^1. & &
\end{eqnarray*}
This is no longer Minkowski space-time, yet there exists a meaningful interpretation for this object, namely the manifold of pairs $(e^{\mathbf{i}\theta}\underline{\omega},-e^{\mathbf{i}\theta}\underline{\omega})$ in $\C^4$, with $e^{\mathbf{i}\theta}\underline{\omega}\in LS^3,$ the Lie sphere. For this manifold the above calculation makes sense. So the problem is not only to know what calculation to make to describe an object correctly but also how to correctly interpret a calculation (morphological synthesis). It often happens that different objects turn out to share the same quantity.

There is no way to avoid these problems; one simply has to experiment until one finds the best fitting calculations or interpretations. This may be seen as a drawback, but we see it as a stronghold that illustrates the richness of the morphological language.
\\

\item[(iii)] CONSISTENCY

Morphological calculus may be compared to making the bill of a meal in a restaurant; usually the bill adds up correctly but sometimes the sum of the ingredients of the meal is more expensive than the meal.

Here is a example in morphological calculus: Consider the space $\R^2_n$ of bivectors in a Clifford algebra: $b=\sum_{i<j}^{}b_{ij}e_ie_j$.

Then $\R^2_n$ is a real vector space of dimension ${n\choose 2}$:
$$\R^2_n = \R^{{n\choose 2}},$$
 but on the other hand , $b\in \R^2_n\setminus\{0\}$ may be written as:
 $$
 b = r_1I_1 + \cdots + r_sI_s, \;\; 2s\leq n,
 $$
whereby $r_1\geq r_2\geq \cdots \geq r_s>0$ is unique and $I_j = \underline{\omega}_j\wedge\underline{\nu}_j$, $|\underline{\omega}_j|=|\underline{\nu}_j| =1$, $|\underline{\omega}_j|\perp|\underline{\nu}_j|$ is a $2$-blade such that $I_jI_k = I_kI_j$, i.e., $(\underline{\omega}_1,\underline{\nu}_1, \ldots, \underline{\omega}_j,\underline{\nu}_s)\in V_{n,2s}(\R)$.

This leads to a partition of $\R^2_n$ into orbits of the orthogonal group $O(r_1,\ldots, r_s)$, which one may calculate morphologically and add up properly.

For $n=3$, $b = r\underline{\omega}\wedge\underline{\nu} \in \R_+ \times \widetilde{G_{3,2}}(\R)$, leading to
$$
\R_3^2 -1 = \widetilde{G_{3,2}}(\R)\cdot\R_+ = S^2\cdot\R_+ =\R^3 - 1,
$$
which adds up correctly. But already for $n=4$ there is a problem. Every $b\in\R^2_4\setminus\{0\}$ may be written as
$$
b = r_1\underline{\omega}_1\wedge\underline{\nu}_1 + r_2\underline{\omega}_2\wedge\underline{\nu}_2,\;\; r_1\geq r_2\geq 0
$$
and there are three cases:
\begin{enumerate}
\item for $r_1>r_2> 0$ the blades $\underline{\omega}_1\wedge\underline{\nu}_1$ and $\underline{\omega}_2\wedge\underline{\nu}_2$ are uniquely determined in terms of $b$, so we have in fact:
\begin{eqnarray*}
   r_2 > 0 & \in & \R_+, \; r_1>r_2\in\R_+,\\
   \underline{\omega}_1\wedge\underline{\nu}_1 & \in & \widetilde{G_{4,2}}(\R) = \frac{S^3\cdot S^2}{S^1} = \BS^2\cdot S^2
\end{eqnarray*}
$[\underline{\omega}_2\wedge\underline{\nu}_2,\underline{\omega}_1\wedge\underline{\nu}_1]=0$ leaves $2$ possibilities: $\underline{\omega}_2\wedge\underline{\nu}_2=\pm \underline{\omega}_1\wedge\underline{\nu}_1\cdot e_{1234}$.

So, in morphological terms we get:
$$
2\BS^2\cdot S^2\cdot\R_+\cdot\R_+ = 2\R_+\cdot\BS^2\cdot(\R^3 -1).
$$
\item for $r_1>0$, $r_2= 0$ we get $b = r\underline{\omega}\wedge\underline{\nu}$ with $r\in\R_+$ and $\underline{\omega}\wedge\underline{\nu}\in \widetilde{G_{4,2}}(\R) = \BS^2\cdot S^2$, so in total $\BS^2\cdot S^2\cdot \R_+ = \BS^2\cdot (\R^3 - 1).$
\item in case $r_1=r_2=r> 0$ we get
$$
b = r(\underline{\omega}_1\wedge\underline{\nu}_1 + \underline{\omega}_2\wedge\underline{\nu}_2),
$$
whereby either $\underline{\omega}_2\wedge\underline{\nu}_2 = \pm e_{1234}\underline{\omega}_1\wedge\underline{\nu}_1$, so $b = r \underline{\omega}\wedge\underline{\nu}(1\pm e_{1234})$. Hereby $\underline{\omega}\wedge\underline{\nu}$ may be chosen to belong to $\widetilde{G_{3,2}}(\R) = S^2$ because in fact every bivector $b\in\R_4^2$ may be decomposed uniquely into self-dual and anti-self-dual parts:
$$
b=\frac{1}{2}(1+e_{1234})b_+ + \frac{1}{2}(1-e_{1234})b_-, \;\; b_{\pm} \in\R_3^2.
$$
So in the above case, $\underline{\omega}\wedge\underline{\nu}\in\widetilde{G_{3,2}}(\R)$ is unique, so that the morphological contribution is given by
$$
2\R_+\cdot S^2 = 2(\R^3 - 1).
$$
\end{enumerate}

Hence, adding up $(1)+(2)+(3)$, we get a total morphological sum of
\begin{eqnarray*}
(2\R_+ + 1)\BS^2\cdot (\R^3-1) + 2(\R^3 - 1) \\
(\R\cdot (\R^2+1)+2)\cdot (\R^3-1) \neq (\R^3 + 1)(\R^3 - 1) = \R^6 -1.
\end{eqnarray*}
The gap in the calculation lies in the difference between $\R\cdot (\R^2+1)+2$ and $\BS^3 = \R^3 + 1.$ If in the above we would replace $\R^2+1=\BS^2$ by $2\R^2+2\R+2=S^2$ we would get a factor $\R\cdot S^2+2 = S^3$ and replacing then $S^2$ by $\BS^2$ would make the bill add up correctly.

So in fact $\R\cdot (\R^2 + 1) + 2$ may be interpreted as an oversized version of the Poincar\'e sphere $\BS^3=\R^3+1.$

Also for the bivector space $\R^2_5$ we have three cases:
\begin{enumerate}
\item in case $r_1>r_2>0$ we obtain the quantity:
$$
\R^2_+\cdot\frac{S^4\cdot S^3}{S^1}\cdot\frac{S^2\cdot S^1}{S^1} = (\R^5-1)\cdot\BS^2\cdot(\R^3-1)
$$

\item in case $r_1>r_2=0$ we obtain:
$$
\R_+\cdot\frac{S^4\cdot S^3}{S^1} = (\R^5-1)\cdot\BS^2
$$
\item in case $r_1=r_2=r>0$ we get bivectors of the form
$r(\underline{\omega}_1\wedge\underline{\nu}_1 + \underline{\omega}_2\wedge\underline{\nu}_2)$ in $\R^5$; the number of choices for $\mathrm{span}\{\underline{\omega}_1,\underline{\nu}_1,\underline{\omega}_2,\underline{\nu}_2\}$ equals $\widetilde{G_{5,1}}(\R) = \frac{S^4}{2}$ while for each choice we have the quantity $2\R_+ S^2$ as before, leading to a total of
$$
\frac{S^4}{2}(2\R_+ S^2) = (\R^5-1)\cdot S^2.
$$
So, the total bill for $\R_5^2$ reads
\begin{gather*}
(\R^5-1)((1+\R^2)\R^3 + 2\R^2 + 2\R + 2)\\
 = (\R^5-1)(\R^5 + (\R+1)(\R^2+\R+1)+1),
\end{gather*}
while we would need the second factor to be equal to $\R^5+1$ to make the bill add up correctly.
\end{enumerate}

From $n\geq 6$ on the calculation of $\R_n^2$ is much more complicated so we won't do it here, but in any case we won't get just $\R^{{n\choose 2}}$. This may be seem as an inconsistency which is likely to repeat itself in cases of partitions of geometrical objects. We have no solution as even explanation of this, but it is clear that one can study this phenomena within the language of morphological calculus, which in itself is consistent.
\\

\item[(iv)] CALCULUS STYLES

A calculus style is obtained by making certain restrictions on the use of the calculus language and by a certain kind of application or focus.

In the canonical style we decided to replace the relation $\R=2\R + 1$ by its more rigorous form  $\R=2\R_+ + 1$ in order to avoid too many unwanted identifications.

This leads to the possibility to apply the rules of calculus on a free basis (commutativity, brackets, etc.) whereby our focus is the calculation of quantity for a large collection of manifolds and this calculation arises from a morphological analysis of the geometrical objects (and constructions) we are interested in.

In the formal style we start off from a given quantity, a polynomial $a_0\R^n + \cdots + a_n$ with $a_0>0,$ $a_1,\ldots,a_n\in\N$ say, and we consider the collection of all the algebraic expressions that evaluate to this quantity. Since we already start with a polynomial with positive integer coefficients, we won't consider any subtractions or divisions here, just addition and multiplication. Also we won't be using $\R_+$ here and the relation $\R=2\R_+ + 1$ will be replaced by a non-commutative and non-associative version of $\R=2\R+1:$
$$
\R=\R+1 + \R,\;\;\; \R=\R + (\R + 1),
$$
the use of which leads to a change in the quantity. Also other calculations involving commuting terms or factors or placing or removing brackets are seen as as morphisms on the collection of morphological objects. So for each quantity we have basically a category.

Parallel to this, for each polynomial $a_0\R^n + \cdots + a_n$ we also have the set of all geometrical objects that can be formed by glueing together (or not) $a_0$ copies of $\R^n,$ $a_1$ copies of $\R^{n-1},\ldots,a_n$ points. The focus now is to study possible correlations between the category of algebraic expressions and geometrical objects (graphs) for a given quantity; this is morphological synthesis.

For example for $\R+1$ we have two expressions
$$
\R + 1,\;\;\; 1 + \R
$$
and two geometrical objects (apart from trivial disjoint union) semi-interval $[0,1[$ or circle $\BS^1$ and one possible correlation is to identify $\R+1$ with $[0,1[$ and $1+\R$ with the circle $\BS^1$.

The more general case $a\R+b$ leads to a kind of calligraphy that we'll study in forthcoming work. For this reason we will speak of calligraphical calculus.
\end{description}

\section{Outlook}

It is not easy to provide complete references to the topic of morphological calculus but certain examples of it as well as related topics are certainly available throughout the mathematical literature. First of all there is our paper \cite{6FS} in which we gave an introduction to morphological calculus which	was subdivided into an axiomatic approach, a canonical part and a formal part based on the formal language of calculus. In this paper we focused mainly on	the canonical part by giving many more new examples of interesting calculations. Morphological calculus can be seen as a formal language and the task of constructing good geometrical interpretations of calculations, called morphological synthesis, can be seen as part of a research field called the theory of Lindenmayer systems (L-systems) for which there is a vast literature. We only refer to \cite{5RS}. Also in the book \cite{4Pen} by Roger Penrose the language of calculus	has been discussed, in particular the meaning of commutativity of the multiplication has been critically investigated. But the present paper is mostly concerned with examples concerning spheres, real and complex projective spaces, special Lie groups and homogeneous spaces including Stiefel manifolds and Gra\ss mannians, various types of complex spheres and real and complex nullcones. All of this belongs to the theory of special manifolds (see e.g. \cite{7War}).

In particular we also discussed real and complex compactified Minkowski spaces as well as twistor spaces which have many applications in mathematical physics and for which we refer to the pioneering work \cite{3PH} of R. Penrose and W. Rindler. Morphological calculus is of course also related to various topics in algebraic topology in particular Betti numbers, homology and cohomology, Poincar\'e polynomials, Euler characteristics and much more that is to be found all over the literature (use Wikipedia and see also \cite{7War}). Finally, many of our calculations also make use of bivector spaces, Clifford algebras and Spin groups for which we refer to the books \cite{1DSS} and \cite{2Lou}.

\section{Acknowledgement}
The author wishes to thank Dr. Narciso Gomes (University of Cape Verde - Uni-CV) for his help in the critical reading and the painstaking task of typewriting this manuscript. We also wish to thank the referee for his valuable suggestions during the preparation of the manuscript.


\end{document}